%% file: Normalizer_Revised.tex
\documentclass[a4paper,10pt]{article}
\usepackage{amsmath, amssymb, amsfonts, amsthm}
\usepackage{mathtools,xcolor}
\usepackage[utf8]{inputenc}
\inputencoding{latin1}
\inputencoding{utf8}
\usepackage{hyperref}
\newtheorem{thm}{Theorem}[section]
\newtheorem{lema}[thm]{Lemma}
\newtheorem{prop}[thm]{Proposition}
\theoremstyle{definition}

\newtheorem{exple}[thm]{Example}

\newtheorem{remark}[thm]{Remark}
\theoremstyle{plain}
\newtheorem{cor}[thm]{Corollary}
\numberwithin{equation}{section}
\usepackage[a4paper, top = 1.2in, bottom = 1.2in, left = 0.8in, right = 0.8in]{geometry}
\numberwithin{equation}{section}

\newcommand{\N}{\mathbb{N}}
\newcommand{\Q}{\mathbb{Q}}
\newcommand{\R}{\mathbb{R}}

\newcommand{\Z}{\mathbb{Z}}

\newcommand{\F}{\mathbb{F}}

\def\1{1\!\!1}

\newcommand{\psmat}[4]{\bigl( \begin{smallmatrix} #1 & #2 \\ #3 & #4 \end{smallmatrix} \bigr)}

\title{The modular automorphisms of quotient modular curves}

\author{ \ Francesc Bars\footnote{First author is supported by PID2020-116542GB-I00 and CEX2020-001084-M} and Tarun  Dalal}

\begin{document}
    \maketitle

     \begin{abstract}
        We obtain the modular automorphism group of any quotient modular curve of level $N$, with $4,9\nmid N$.
        In particular, we obtain some non-expected automorphisms of order 3 that { appear for} the quotient modular curves when the Atkin-Lehner involution $w_{25}$ belongs to the quotient modular group, such { automorphisms are} not necessarily defined over $\mathbb{Q}$. As a consequence of the results, we obtain the full automorphism group of the quotient modular curve $X_0^*(N^2)$, for sufficiently large $N$.
    \end{abstract}
\input{IntroductionNormalizer_Revised}
    \input{Section1FinalVersion_Revised}
\input{Exact_Normalizer_Revised} 
    \input{Section3FinalVersion_Revised}

    \section*{Acknowledgments}
The first author's work is supported by the Spanish
State Research Agency through the Severo Ochoa and María de Maeztu Program for Centers and Units
of Excellence in R\&D (CEX2020-001084-M), alongside PID2020-116542GB-I00, Ministerio de Ciencia y
Universidades of Spanish government.
Some part of the paper was written during the second author's visit to the Universitat Aut\`{o}noma de Barcelona and the second author is grateful to the Department of Mathematics of Universitat Aut\`{o}noma de Barcelona for its support and hospitality.

\bibliographystyle{alpha}

\noindent{Francesc Bars Cortina}\\
{Departament Matem\`atiques, Edif. C, Universitat Aut\`onoma de Barcelona\\
08193 Bellaterra, Catalonia;}\\
Centre de Recerca Matemàtica (CRM),
 C. dels Til.lers,\\
  08193 Bellaterra, Catalonia, Spain\\
{Francesc.Bars@uab.cat}

\vspace{1cm}

\noindent{Tarun Dalal}\\
 {Institute of Mathematical Sciences, ShanghaiTech University\\ 393 Middle Huaxia Road, Pudong, Shanghai 201210, China}\\
 {tarun.dalal80@gmail.com}

\end{document}

%% file: IntroductionNormalizer_Revised.tex
\begin{section}{Introduction}
    A curve $C$ with non-trivial automorphism group encodes deep arithmetic information (in particular the twists of the curve $C$). The Fermat quartic or the Klein quartic are examples
     of curves with big automorphism groups extensively studied in the literature.

    Some of the main curves in arithmetic geometry are the classical modular curves $X$ over $\Q$, which are moduli spaces classifying elliptic curves with some $N$-level structure. A non-trivial automorphism group would have deep arithmetic meaning for such curves. For example, when $N$ is prime, one expects that such a modular curve $X$  has no rational { point} except the cusps and points associated to elliptic curves with
    complex multiplication, usually called CM points, (which is related with Serre's uniformity conjecture). In \cite{Do16}, the author related the existence (for certain $X$) of non-trivial automorphisms
     with the existence of rational points that are neither CM nor cusp. Thus modular curves with non-trivial automorphism group are of key interest.

    Let $X$ be a modular curve (we assume that it is defined over $\Q$), where its complex points { correspond to} the completion at certain cusps of the upper half plane $\mathcal{H}$ modulo the action by a congruence subgroup
    $\Gamma\leq \mathrm{SL}_2(\mathbb{Z})$ (we assume $\pm I\in\Gamma$), and denote by $\mathrm{Aut}(X)$ the automorphism group over $\overline{\mathbb{Q}}$ of the modular curve { $X$}. In particular, the normalizer of $\Gamma$ in
    $\mathrm{PSL}_2(\R)$ (the automorphism group of $\mathcal{H}$) modulo $\Gamma$ provides a subgroup of $\mathrm{Aut}(X)$ {which is known as  {\bf the modular automorphism group} of $X$.}
    For a group $G\leq \mathrm{PSL}_2(\R)$, we denote its normalizer inside $\mathrm{PSL}_2(\R)$ by $\mathcal{N}(G)$. {In particular, the modular automorphism group of the modular curve associated to $\Gamma$ corresponds to $\mathcal{N}(\Gamma)/\Gamma$.}

        Let $N\in \N$ (where $\N$ denotes the set of all positive integers), and consider the modular group $\Gamma_0(N):=\{\psmat{a}{b}{Nc}{d}\in \mathrm{SL}_2(\Z)\}$.
       It is well known that the associated modular curve $X_0(N)$ is defined over $\mathbb{Q}$.
    Atkin-Lehner in \cite[Theorem 8]{AtLe}, stated the result for $\mathcal{N}(\Gamma_0(N))$ modulo $\Gamma_0(N)$, (cf. \cite{AkSin}, \cite{Ba} for the correct statement and the proof of the result). Such normalizer contains
     the Atkin-Lehner involutions defined by the matrices of the form $w_{d,N}=\frac{1}{\sqrt{d}}\psmat{dx}{y}{Nz}{dw} 
    \in \mathrm{SL}_2(\mathbb{R})$ with $d>1,d||N$ (i.e., $d|N$ and $(d,N/d)=1$) and $x,y,z,w\in\mathbb{Z}$ such that $x w d-yz(N/d)=1$ (we also use the notation $w_d$ to denote $w_{d,N}$, the level $N$ will be clear from the context).
    We denote the group generated by all such Atkin-Lehner involutions modulo $\Gamma_0(N)$ by $B(N)$, which is an abelian group with every non-trivial element of order $2$. For $4,9\nmid N$, we known that 
    $$\mathcal{N}(\Gamma_0(N))/\Gamma_0(N)=B(N),$$
    a group of order $2^{\omega(N)}$, where $\omega(N)$ is the number of distinct prime divisors of $N$ (loc.cit.). Later, Conway \cite{Conway} gave a characterization of the normalizer of $\Gamma_0(N)$ in terms of a group action on lattices,
    which has deep interest and consequences in Group Theory.
     We emphasize here that the existence of such Atkin-Lehner automorphisms (involutions) play a crucial role in the understanding of the modular curves $X_0(N)$ and the theory of Hecke operators for $X_0(N)$ (cf. \cite{AtLe}).

    Now consider any subgroup $W_N$ of $B(N)$ (by abuse of notations we denote the collection of distinct representatives of $B(N)$ by $B(N)$), and the group  $\langle \Gamma_0(N), W_N\rangle$ { (we use the notation $H+G$ or $\langle H,G\rangle$ to denote the group generated by the elements of $H$ and $G$)}.
The associated modular curve $X_0(N)/W_N$ is known as a quotient modular curve and it is defined over $\Q$.
    An automorphism of $X_0(N)/W_N$ is said to be \textbf{modular} if it is coming from an element of $\mathrm{PSL}_2(\R)$ (note that such an element belongs to $\mathcal{N}(\langle \Gamma_0(N), W_N\rangle)$ and vice-versa). 
    Since $\mathcal{N}(\Gamma_0(N))/\Gamma_0(N)=B(N)$ for $4,9\nmid N$, it is natural to ask whether the equality $\mathcal{N}(\langle \Gamma_0(N), W_N\rangle)/\langle \Gamma_0(N), W_N\rangle=B(N)/W_N$ is true or not for $4,9\nmid N$.
    When $N$ is square-free, Lang in \cite{Lang} proved that the equality $\mathcal{N}(\langle \Gamma_0(N), W_N\rangle)/\langle \Gamma_0(N), W_N\rangle=B(N)/W_N$ is true for any subgroup $W_N$.
    The main motivation of this article is to study this question for general $N$ with $4,9\nmid N$. More precisely, we completely determine the normalizer $\mathcal{N}(\langle \Gamma_0(N), W_N\rangle)$ and prove the following results.
    \begin{thm}\label{main theorem 1} [Theorem \ref{Complete normalizer without 25 theorem} in text]
        Let $N\in \N$ and $W_N$ be a subgroup generated by the Atkin-Lehner involutions such that $4,9\nmid N$ and $w_{25}\notin W_N$. Then
        $\mathcal{N}(\langle \Gamma_0(N), W_N\rangle) = \langle \Gamma_0(N),{ w_{d}}: d||N\rangle.$
    \end{thm}
    \begin{thm}\label{main theorem 2} [Theorem \ref{Complete normalizer with 25 theorem} in text]
        Let $N\in \N$ and $W_N$ be a subgroup generated by the Atkin-Lehner involutions such that $4,9\nmid N$ and $w_{25}\in W_N$.
        \begin{enumerate}
            \item If there exists $w_d\in W_N$ such that $\frac{d}{(25,d)}\not \equiv \pm 1 \pmod 5$, then $\mathcal{N}(\langle \Gamma_0(N), W_N\rangle) = \langle \Gamma_0(N),{ w_{d}}: d||N\rangle$.
            \item If $\frac{d}{(25,d)}\equiv \pm 1 \pmod 5$ for all $w_d\in W_N$, then $\mathcal{N}(\langle \Gamma_0(N), W_N\rangle) = \langle \Gamma_0(N), \Upsilon_5^{-1} B_jC_0\Upsilon_5, \Upsilon_5^{-1}B_0C_i\Upsilon_5, w_{d}: d||N\rangle$
             where $\Upsilon_5:=\psmat{1}{0}{0}{1/5}$, $B_j:=\psmat{\frac{N}{25} j+1}{-j}{-{\frac{N}{25}}}{1}$, $C_i:=\psmat{1}{i}{0}{1}$, $0\leq j,i \leq 4$ such that  ${\frac{N}{25}}j\equiv 2\pmod 5$ and $i\equiv -j \pmod 5$.         
             { Moreover $\langle \Upsilon_5^{-1} B_jC_0\Upsilon_5 = (\Upsilon_5^{-1}B_0C_i\Upsilon_5)^{-1} \rangle$ has order $3$ in $ \mathcal{N}(\langle \Gamma_0(N), W_N\rangle)/\langle \Gamma_0(N), W_N\rangle$.}             
        \end{enumerate}
    \end{thm}
    As an immediate consequence of Theorem \ref{main theorem 1} and Theorem \ref{main theorem 2}, we obtain that { for $4,9\nmid N$,} we have $\mathcal{N}(\langle \Gamma_0(N), W_N\rangle)/\langle \Gamma_0(N), W_N\rangle \supsetneq B(N)/W_N$
    if and only if $w_{25}\in W_N$ and $\frac{d}{(25,d)}\equiv \pm 1\pmod 5$ for all $w_d\in W_N$. Moreover, in such cases the group $\mathcal{N}(\langle \Gamma_0(N), W_N\rangle)/\langle \Gamma_0(N), W_N\rangle$
    (and hence the group $\mathrm{Aut}(X_0(N)/W_N)$) may be non-abelian and contains elements of order $3$.
    In particular, this explains the new automorphisms of order 3 that appear for the quotient curves $X_0(25 q)/\langle w_{25}\rangle$ and $X_0(25 q)/\langle w_{25},w_q\rangle$ with $q$ prime (under some assumptions),
     which is first observed in \cite{BaTa24}.

    It is expected that when $N$ is sufficiently large, the modular automorphism group of $X_0(N)/W_N$ coincides with the full automorphism group $\mathrm{Aut}(X_0(N)/W_N)$. This statement is true for the modular curve $X_0(N)$ (cf. \cite{KeMo}). Moreover, when $N$ is either square-free (cf. \cite{BG21}) or a perfect square (cf. \cite{DLM22}), then this statement is true for the modular curve $X_0^*(N):=X_0(N)/B(N)$.
    In particular, combining Theorem \ref{main theorem 1} and Theorem \ref{main theorem 2} with \cite[Theorem 5.8]{DLM22} we get
        \begin{cor}
            Let $N\geq 10^{400}$ and $(6,N)=1$. Then
               $\mathrm{Aut}(X_0(N^2)/B(N^2))\cong \begin{cases}
                    \Z/3\Z, \ \mathrm{if \ } 5||N,\\
                    \{\mathrm{id}\}, \ \mathrm{otherwise}.
                \end{cases}$
        \end{cor}

    In the last section of this paper, under some assumption we prove that { the order $3$ modular automorphisms} are defined over $\mathbb{Q}(\sqrt{5})$.

    \end{section}

%% file: Section1FinalVersion_Revised.tex
    \section{The Conway Big Picture for quotient modular groups}
 	For $N\in \N$ and a subgroup $W_N$ generated by certain the Atkin-Lehner involutions, consider the group $\langle \Gamma_0(N),W_N\rangle$.
	 We denote by $\Gamma_0^*(N)$, the subgroup generated by $\Gamma_0(N)$ and all the Atkin-Lehner involutions $w_{d,N}$ with $d||N$. The aim of this section is to prove that $\mathcal{N}(\langle \Gamma_0(N),W_N\rangle)$ is a subgroup of $\Gamma_0^*(M)$, for some positive divisor $M$ of $N$. In order to do this, we will follow Conway's Big Picture introduced in \cite{Conway}.
%
%
%
%
%
%
%
    \subsection{The Big Picture}

{{
Two lattices $L(1)$ and $L(2)$ ({commensurable} with
$\mathbb{Z}\times\mathbb{Z}$) are equivalent to each other if there
exists $q\in\mathbb{Q}^*$ such that $L(1)=q L(2)$. This is an equivalence relation on the set of lattices that are commensurable with $\Z\times\Z$.
Each equivalence class has a representative of the form $L_{s,g/t}:=\langle
(s,g/t),(0,1)\rangle_{Lat}=\langle s e_1+\frac{g}{t}e_2,e_2\rangle_{Lat}$, where $s>0$ is a rational number and
$0\leq g/t<1$, with $g\geq 0$ and $t>0$ coprime integers, $e_1=(1,0)$ and $e_2=(0,1)$; when $g=0$ we
denote $L_s$ by $L_{s,0}$. For simplicity of notations we denote the equivalence
class containing the lattice $L_{s,g/t}$ by $L_{s,g/t}$. The hyperdistance between two equivalence
classes $L_{s_1,g_1/t_1}$ and $L_{s_2,g_2/t_2}$ {  is defined as follows: 
after a suitable base change one class corresponds to $\langle
e_1,e_2\rangle_{Lat}$ and the other corresponds to $\langle k e_1,
e_2\rangle_{Lat}$ for a certain $k\in \N$, the number $k$ is the hyperdistance between $L_{s_1,g_1/t_1}$ and
$L_{s_2,g_2/t_2}$.}

The Big Picture of Conway is a graph defined as follows: the points (or vertices) correspond
to the equivalence classes $L_{s,g/t}${ ,} and two classes
$\{L_{s_1,g_1/t_1},L_{s_2,g_2/t_2}\}$ are {connected by a non-oriented edge if and only if the hyperdistance between such two classes is a prime number}.}}

There is a natural action of $\mathrm{PGL}_2(\mathbb{Q})$ on the Big Picture defined as follows: for $A:=\psmat{a}{b}{c}{d}
	\in \mathrm{GL}_2(\mathbb{Q})$, 
$A*L_{s,g/t}$ { corresponds} to the representative of the class containing the lattice $\langle s(a e_1+b e_2)+\frac{g}{t}(c e_1+d e_2), ce_1+d e_2\rangle_{Lat}$, i.e., in terms of basis elements, this action can be written as:
%
\begin{equation} \label{eq1}
    s e_1+\frac{g}{t}e_2 \mapsto s (a e_1+b e_2)+\frac{g}{t}(ce_1+de_2), \ 
    e_2 \mapsto ce_1+de_2.
\end{equation}
This action could be extended to $\mathrm{PSL}_2(\mathbb{R})$ with the same definition.
 We would like to remark that $w_{d,N}\in \mathrm{PSL}_2(\mathbb{Q})$ when $d$ is a perfect square, and $w_{d,N}\notin \mathrm{PSL}_2(\mathbb{Q})$ if $d$ is not a perfect square.
%
%

The following results are well-known (cf \cite{Conway}).
\begin{thm} [Conway] The stabilizer of $X=L_{s,g/t}$ in $\mathrm{PSL}_2(\R)$ is     $\psmat{s}{g/t}{0}{1}^{-1}\mathrm{PSL}_2(\mathbb{Z})\psmat{s}{g/t}{0}{1}\subseteq \mathrm{PSL}_2(\mathbb{Q}),$
 {    and in the Big Picture literature such stabilizer is denoted
$\Gamma_0(X|X)+$.}
\end{thm}

Following the notation in the Big Picture, for {a positive integer $h$ with $h^2|N$,} we define the group
$$\Gamma_0(N/h|h)+:=\psmat{1/h}{0}{0}{1} \Gamma_0^*(N/{ h^2}) \psmat{h}{0}{0}{1}\subset \mathrm{PGL}_2(\Q).$$

\begin{thm}[Conway]\label{Conway}
    The point $L_{s,g/t}$ is fixed by $\Gamma_0(N)$ if and only if {$s$ is a positive integer and $t|24$ is the largest integer such that $t^2|N$ and $1|s|(N/t^2)$. 
The collection of all such points on the Big Picture is called the $(N|1)$-snake.}
   Furthermore, $\sigma\in \mathrm{PSL}_2(\mathbb{R})$ leaves the $(N|1)$-snake invariant as a set (not { necessarily point-wise}) if and only if $\sigma \in\Gamma_0(N/h|h)+$, where $h$ is the largest { positive integer} such that $h|24$ and $h^2|N$. Thus { $\mathcal{N}(\Gamma_0(N))=\Gamma_0(N/h|h)+$.}
\end{thm}
\begin{exple}
    If $4,9\nmid N$, then from Theorem \ref{Conway} the $(N|1)$-snake corresponds to the set of all classes $L_{s,0}$ where $s$ is a { positive divisor of $N$}.
Furthermore, under such assumption we have $\mathcal{N}(\Gamma_0(N))=\Gamma_0(N/1|1)+=\Gamma_0^*(N)$.
\end{exple}

\subsection{The normalizer of $\Gamma:=\Gamma_0(N)+W_N$ { with} $w_{u^2,N}\notin\Gamma$.}\label{subsection 2.2}
Let $N$ be a positive integer. Following the ideas of \cite{Lang}, we now study the normalizer of $\Gamma_0(N)+W_N$, where $W_N$ is a subgroup generated by Atkin-Lehner involutions such that for any Atkin-Lehner involution $w_{d,N}\in \Gamma_0(N)+W_N$, $d$ is not a perfect square. Note that under such assumption $w_{d,N}\notin \mathrm{PSL}_2(\mathbb{Q})$. For simplicity of notations we write $w_{d,N}$ by $w_d$. Let $d_1,\ldots,d_n$ be exact divisors of $N$ (i.e., $d_i||N$) such that $\Gamma_0(N)+W_N=\Gamma_0(N)+\langle w_{d_1},\ldots,w_{d_n}\rangle$. Now consider the action of $\Gamma_0(N)+W_N$ on the Big Picture.
%
We use the following notations:
\begin{itemize}
    \item \texttt{t}: a $(\Gamma_0(N)+\langle w_{d_1},\cdots, w_{d_n}\rangle)$-orbit of size $2^n$,
    \item $T_N$: the set of all the \texttt{t}'s.
\end{itemize}
\begin{lema}\label{number of elements in a +we1,..., wen orbit}
  Let $N$ be a positive integer, and $W_N$ be a subgroup generated by Atkin-Lehner involutions such that for any Atkin-Lehner involution $w_d\in \Gamma_0(N)+W_N$, $d$ is not a perfect square. For each $X\in (N|1)$-snake
  $$\{\sigma(X): \sigma \in (\Gamma_0(N)+\langle w_{d_1},\cdots, w_{d_n}\rangle)\}=\{X,w_d(X): d||N, w_d\in \Gamma_0(N)+W_N\}$$
  is a member of $T_N$.
\end{lema}
\begin{proof}
    Let $X\in (N|1)$-snake. Then $X$ is fixed by $\Gamma_0(N)$. 
{ Since $w_{d_i}\in \mathcal{N}(\Gamma_0(N))$, the elements of $\Gamma_0(N)+\langle w_{d_1},\cdots, w_{d_n}\rangle$ are of the form $\prod_{i=1}^n w_{d_i}^{k_i}\gamma$ for some $\gamma\in \Gamma_0(N)$ and $k_i\in \{0,1\}$. Moreover, we can write }    
    $\prod_{i=1}^n w_{d_i}^{k_i}\gamma=w_d^k\gamma^\prime, \ \mathrm{for \ some \ } \gamma^\prime\in \Gamma_0(N), \ w_d\in W_N \ \mathrm{and} \ k\in \{0,1\}.$ Hence
$$\{\sigma(X): \sigma \in (\Gamma_0(N)+\langle w_{d_1},\cdots, w_{d_n}\rangle )\}=\{X, w_d(X): d||N,w_d\in\Gamma_0(N)+W_N\}.$$
   Let $\mathcal{C}_{W_N}$ denote the representatives of distinct left cosets of $\Gamma_0(N)$ in $\Gamma_0(N)+W_N$.
    Since $w_d^2\in \Gamma_0(N)$ for $d||N$, $[\Gamma_0(N)+W_N:\Gamma_0(N)]=2^n$ and $\Gamma_0(N)$ fixes $X$, the set $\{ X, w_d(X): d||N, w_d\in\Gamma_0(N)+W_N\}=\{X,\delta(X): \delta \in \mathcal{C}_{W_N}\setminus\{\mathrm{id}\}\}$ has at most $2^n$ elements. Recall that the stabilizer of $X$ is of the form $(\Gamma_0(X|X)+)\subseteq \mathrm{PSL}_2(\Q)$.
    Let $\delta_1,\delta_2$ be two distinct elements of $\mathcal{C}_{W_N}\setminus \{\mathrm{id}\}$. { Then there exist integers $d_{m_1},d_{m_2}$ with $d_{m_i}||N$ such that $\delta_i=w_{d_{m_i}}\in W_N$ for $i\in \{1,2\}$.} By the assumption on $W_N$, it is easy to see that $w_{d_{m_1}},w_{d_{m_2}},w_{d_{m_1}}^{-1}w_{d_{m_2}}\not\in \mathrm{PSL}_2(\Q)$. Thus $\delta_i(X)\neq X$ and $\delta_1(X)\neq \delta_2(X)$. 
    The result follows.
\end{proof}
\begin{lema}\label{orbit is a subset of a snake}
    Under the assumptions of Lemma \ref{number of elements in a +we1,..., wen orbit}, for any $\texttt{t}\in T_N$, $\texttt{t}$ is a subset of the $(N|1)$-snake.
\end{lema}
\begin{proof}
    Consider an element $\texttt{t}\in T_{N}$, and $X=L_{s,g/t}\in\texttt{t}$. Suppose $X\notin (N|1)$-snake. Then $\Gamma_0(X|X)+$ is not a supergroup of $\Gamma_0(N)$ (cf. \cite[p.33,(8)]{Lang}). Consequently, $\Gamma_0(N)\cap\Gamma_0(X|X)+$ is a proper subgroup of $\Gamma_0(N)$. In particular we have
    $$[\Gamma_0(N):\Gamma_0(N)\cap\Gamma_0(X|X)+]\geq 2.$$
        Recall that $\Gamma_0(X|X)+\subseteq \mathrm{PSL}_2(\mathbb{Q})$, and $(\Gamma_0(N)+\langle w_{d_1},\cdots, w_{d_n}\rangle)\cap \mathrm{PSL}_2(\mathbb{Q})=\Gamma_0(N)$. Hence we obtain
       \begin{align*}(\Gamma_{0}(N)&+\langle w_{d_1},\cdots, w_{d_n}\rangle )\cap\Gamma_0(X|X)+=\Gamma_0(N)\cap \Gamma_0(X|X)+.\\
 \mathrm{Therefore,} \hspace*{2cm} &[\Gamma_0(N)+\langle w_{d_1},\cdots, w_{d_n}\rangle :(\Gamma_0(N)+\langle w_{d_1},\cdots, w_{d_n}\rangle )\cap \Gamma_0(X|X)+]\\
   =&[\Gamma_0(N)+\langle w_{d_1},\cdots, w_{d_n}\rangle :\Gamma_0(N)\cap \Gamma_0(X|X)+]\\
   =&[\Gamma_0(N)+\langle w_{d_1},\cdots, w_{d_n}\rangle :\Gamma_0(N)][\Gamma_0(N):\Gamma_0(N)\cap\Gamma_0(X|X)+]\geq 2^{n+1}.
  \end{align*}
         Observe that $(\Gamma_0(N)+\langle w_{d_1},\cdots, w_{d_n}\rangle)\cap \Gamma_0(X|X)+$ is the stabilizer of $X$ in $\Gamma_0(N)+\langle w_{d_1},\cdots, w_{d_n}\rangle $. The last equality shows that the $(\Gamma_0(N)+\langle w_{d_1},\cdots, w_{d_n}\rangle )$-orbit of $X$ has at least $2^{n+1}$ elements, which contradicts that $X\in \texttt{t}$.
         Therefore $X\in (N|1)$-snake.
\end{proof}
\begin{lema}\label{lemma 2.6 section 2}
Let $N,W_N$ be as in lemma \ref{number of elements in a +we1,..., wen orbit}.
    Then, $\mathcal{N}(\Gamma_0(N)+W_N)$ is a subgroup of $\Gamma_0(N/h|h)+$, where $h|24$ is the largest natural number such that $h^2|N$.
\end{lema}
\begin{proof}
    By Lemma \ref{number of elements in a +we1,..., wen orbit} and Lemma \ref{orbit is a subset of a snake}, $X\in (N|1)$-snake if and only if $X\in \texttt{t}\in T_N$.

    Now for each $\sigma\in \mathcal{N}(\Gamma_0(N)+\langle w_{d_1},\cdots, w_{d_n}\rangle)$, and $\texttt{t}\in T_N$, we have
    $$(\Gamma_0(N)+\langle w_{d_1},\cdots, w_{d_n}\rangle)\sigma(\texttt{t})=\sigma(\Gamma_0(N)+\langle w_{d_1},\cdots, w_{d_n}\rangle)(\texttt{t})=\sigma(\texttt{t}).$$
Thus $\sigma$ fixes the $(N|1)$-snake. Now the result follows from the fact that $\sigma\in \mathrm{PSL}_2(\R)$ leaves the $(N|1)$-snake invariant if and only if $\sigma\in \Gamma_0(N/h|h)+$, where $h|24$ is the largest natural number such that $h^2|N$ (cf. Theorem \ref{Conway}).
\end{proof}

\begin{cor} 
Let $N,W_N$ be as in Lemma \ref{number of elements in a +we1,..., wen orbit} with $4,9\nmid N$.
    Then, { $\mathcal{N}(\Gamma_0(N)+W_N)=\Gamma_0^*(N)$}. In particular the modular automorphism group of the quotient curve $X_0(N)/W_N$ is $B(N)/W_N$.
\end{cor}
\begin{proof}
    Under the assumption $4,9\nmid N$ we have $h=1$ in Lemma \ref{lemma 2.6 section 2}. Now the result follows from the facts that $\Gamma_0(N/1|1)+=\Gamma_0^*(N)$ and $\mathcal{N}(\Gamma_0(N)+W_N)\supseteq \Gamma_0^*(N)$.
\end{proof}

\subsection{Towards the normalizer of $\Gamma:=\Gamma_0(N)+W_N$ with $w_{u^2,N}\in\Gamma$.}

{{
Next consider the group $\Gamma:=\Gamma_0(N)+W_N$ such that $w_{u^2,N}\in \Gamma\setminus\{\mathrm{id}\}$ for some natural number $u\neq 1$. Inspired by \cite{Conway} we try to obtain the points $L_{s,g/t}$ of the Big Picture which are fixed by $\Gamma$.

Consider the conjugation by $\Upsilon_u=\psmat{1}{0}{0}{1/u}$ of $\Gamma_0(N)+W_N$, where we write once and for all in this subsection $N=M \cdot u^2$ with $(M,u^2)=1$.
\begin{lema}\label{lemabis}
    The conjugation by $\Upsilon_u$ satisfies the following properties:
    \begin{itemize}
        \item $\Upsilon_u \Gamma_0(N)\Upsilon_u^{-1}=\{\psmat{a}{b}{c}{d}\in\Gamma_0(M u)|b\equiv 0\pmod u\}$, and we denote such conjugate group by $\tilde{\Gamma}^u_0(M u)$,
\item 
$\Upsilon_u w_{u^2,N}\Gamma_0(N)\Upsilon^{-1}_u=\{\psmat{a}{b}{Mc}{d}\in\Gamma_0(M)|a\equiv d \equiv 0 \pmod u\}.$
    \item If $(d,u)=1$, then $\Upsilon_u w_{d,N}\Upsilon_u^{-1}$ is equal to $w_{d,M u}$ or $w_{d,M}$. If $(d,u^2)=u'$, then $\Upsilon_u w_{d,N}\Upsilon_u^{-1}$ is equal to $w_{d/u',M u}$ also equal to $w_{d/u',M}$.
    \end{itemize}
\end{lema}
We write $\tilde{\Gamma}_{uM}^u:=\Upsilon_u(\Gamma_0(N)+\langle w_{u^2,N}\rangle)\Upsilon_u^{-1}=\langle \tilde{\Gamma}_{u}(M), \tilde{\Gamma}_{0}^u(uM)\rangle$. We now study the lattices $L_{s,g/t}$ fixed by $\tilde{\Gamma}^u_{u M}$. Note that $\tilde{\Gamma}^u_{u M}$ fixes the class containing the lattice $L_{s,g/t}$ if and only if it fixes the lattice $L_{s,g/t}$.


\begin{lema}\label{lema1}
If the equivalence class $L_{s,g/t}$ is fixed by  $\tilde{\Gamma}^u_0(M u)$, then 
$s u\in \mathbb{Z}$, and $s u t^2$ is a divisor of $u^2 M$.
\end{lema}
\begin{proof}
 An arbitrary element of $\tilde{\Gamma}^u_0(M u)$ can be written in the form $\psmat{a}{ub}{Muc}{d}$
		with $a,b,c,d \in \Z$ such that $ad-Mu^2bc=1$. 
		Assume that the lattice $L_{s,g/t}$ is fixed by $\tilde{\Gamma}^u_0(M u)$.
	Recall that $L_{s,g/t}$ are generated by the vectors $v_1:=s e_1+\frac{g}{t}e_2$ and $v_2:=e_2$, and the smallest multiple of $e_1$ that it contains is $v_3= s t e_1=t v_1-g v_2$. Under the action \eqref{eq1}, the matrix $\psmat{1}{u}{0}{1}$
	sends the lattice $L_{s,g/t}$ to the lattice generated by  $v_1+ s u v_2$ and $v_2$. Since $\psmat{1}{u}{0}{1}$ fixes $L_{s,g/t}$, we must have $su\in \Z$.

For the second condition, consider the matrix $\psmat{1}{0}{-Mu}{1}$
	which sends the lattice $L_{s,g/t}$ to the lattice generated by { $v_1'=v_1+\frac{g}{t}(-M u) e_1$} and $v_2'=v_2-M u e_1$. Since $\psmat{1}{0}{-Mu}{0}$ fixes the lattice $L_{s,g/t}$, we have $L_{s,g/t}=\langle v_1',v_2'\rangle_{Lat}$. In particular, this implies $Mue_1,\frac{g}{t}Mue_1\in L_{s,g/t}$. Since $ste_1$ is the smallest multiple of $e_1$ which belongs to $L_{s,g/t}$, there exist $k_1,k_2\in \Z$ such that $gMu^2=sut^2k_1$ and $Mu^2=stuk_2$. Since $(g,t)=1$, these relations give  $sut^2|Mu^2$ as claimed. 
\end{proof}
\begin{lema}
\label{fixed lattices of 1 Atkin-Lehner involution}
Let $M,u\in \N$ such that $(u,M)=1$ and $4,9\nmid N$, where $N=Mu^2$. If $\tilde{\Gamma}^u_{M u}$ fixes the lattice $L_{s,\frac{g}{t}}$, then $L_{s,\frac{g}{t}}$ is of the form $L_{d,0}$ where $d$ is positive divisor of $M$.
\end{lema}
\begin{proof}
  Suppose $\tilde{\Gamma}^u_{u M}$ fixes the lattice $L_{s,\frac{g}{t}}$. By Lemma \ref{lema1} we have $s u\in\mathbb{Z}$ and $ s u t^2| u^2 M$. For simplicity of notations, in the proof we denote the group $\tilde{\Gamma}^u_{u M}$ by $\tilde{\Gamma}$.
	
	Since $(u,M)=1$, there exist $x,y\in \Z$ such that $u^2 x-M y=1$, i.e., { $\psmat{u x}{y}{M}{u}\in \tilde{\Gamma}$}. Since $\psmat{u x}{y}{M}{u}$ fixes $L_{s,\frac{g}{t}}$, we must have
	$M e_1\in L_{s,\frac{g}{t}}$. Thus there exist $c_1,c_2\in \Z$ such that 
	\begin{equation}\label{eq 1.3 new}
		M e_1=c_1(se_1+\frac{g}{t}e_2)+c_2e_2.
	\end{equation}
	From \eqref{eq 1.3 new}, we get $M=c_1s$ and { $c_1\frac{g}{t}\in \Z$}. Since $(g,t)=1$, we must have $t|c_1$ and there exists $N_1\in \Z$ such that $M=s t N_1$.
	Recall that if $t=1$, then $g=0$. Now assume that $g\ne 0$ equivalently $t>1$. 
	
	\underline{\textbf{Case I}}: First assume that $s$ is an integer. 
	
	The condition $M=s t N_1$ implies that $t|M$, $s|M$, and $(u,t)=(u,s)=1$ (recall that $N=Mu^2$ and $(u,M)=1$ by assumption). Furthermore, the condition $s u t^2| u^2 M$ implies that $t^2|M$. {{Since $4,9\nmid N$ there exists a prime $\ell\geq 5$ with $\ell|t|M$. Moreover, we can choose $w\in \Z$ such that $u^2w^2\not\equiv 1 \pmod \ell$ (this is possible since $(u,t)=1$ and $(\Z/\ell \Z)^\times$ has an element of order more than $2$).
	There exist $x,y,z\in\mathbb{Z}$ such that $u^2xw-Mtyz=1$ i.e., $\psmat{u x}{y}{M t z}{u w}\in \tilde{\Gamma}$.}}
	Since $\psmat{u x}{y}{M t z}{u w}$ fixes the lattice $L_{s,\frac{g}{t}}$, we get 
	\begin{equation}
		s(u xe_1+ye_2)+\frac{g}{t}(M t ze_1+u we_2)\in L_{s,\frac{g}{t}}.
	\end{equation}
	Since $s\in \Z$ and $e_2\in L_{s,g/t}$, the last equation implies
	\begin{equation}
		s u x e_1+\frac{g}{t}u w e_2\in L_{s,\frac{g}{t}}.
	\end{equation}
	Thus there exist $c_1,c_2\in \Z$ such that 
	\begin{equation}
		s u x e_1+\frac{g}{t}u w e_2=c_1(s e_1+\frac{g}{t}e_2)+c_2 e_2.
	\end{equation}
	Solving the above equation we get $c_1=u x\in \Z, \ c_2=\frac{g}{t} u (w-x)\in \Z$.
	Since $(g,t)=1$ and $\frac{g}{t} u (w-x)\in \Z$, we must have $ux\equiv uw \pmod t$, in particular we have $ux\equiv uw {\pmod \ell}$. On the other hand, the relation $u^2xw-Mtyz=1$ implies that $u^2w^2\equiv 1 \pmod \ell$, which contradicts the assumption that $u^2w^2\not\equiv 1 \pmod \ell$. Therefore $c_2\not\in \Z$. Hence $s$ can not be an integer if $t>1$.

	
	\underline{\textbf{Case II}}: Now suppose that $s$ is not an integer. From the relation $su\in \Z$, it is clear that $s\in \Q$. Let $p$ be a prime such that $v_p(s)<0$ (for a prime $p$ and $n\in \N$, we use the notation $v_p(n)$ to denote the unique integer $n_p$ such that $p^{n_p}||n$).
	 Since $su\in \Z$, we have $v_p(u)>0$ and $v_p(s)+v_p(u)\geq 0$.
	
	There exist $a,c,d\in \Z$ such that $u^2ad-Mc=1$, { i.e.,} $\psmat{ua}{1}{Mc}{ud}\in \tilde{\Gamma}$. Since $L_{s,\frac{g}{t}}$ is fixed by $\psmat{ua}{1}{Mc}{ud}\in \tilde{\Gamma}$, we have $s(u ae_1+e_2)+\frac{g}{t}(Mc e_1+u d e_2)\in L_{s,\frac{g}{t}}$. Hence there exist $c_{11},c_{12}\in \Z$ such that 
	\begin{equation}
		s(u ae_1+e_2)+\frac{g}{t}(Mc e_1+u d e_2)= c_{11}(s e_1+\frac{g}{t}e_2)+c_{12}e_2.
	\end{equation}
	Solving the last equation we get
	\begin{equation}\label{value of c11,c12 for general level}
		c_{11}=u a+gN_1 c\in \Z, \ c_{12}=s+\frac{g}{t}(u d-u a-g N_1 c)\in \Z.
	\end{equation}
	If $v_p(t)<-v_p(s)$, then from \eqref{value of c11,c12 for general level} it is easy to see that { $v_p(c_{12})=v_p(s)<0$}. Hence $v_p(t)\geq -v_p(s)$, consequently from the relations $M=stN_1$ and $(M,u)=1$ we get $v_p(t)=-v_p(s)$ and $v_p(N_1)=0$. 
	
	Since $v_p(u)>0$, the assumption $4,9\nmid N$ forces that $p\geq 5$.
	In this case, there exists $y\in \Z$ such that {$(u,y)=1$ and $y^2\not\equiv 1 \pmod {p^{-v_p(s)}}$ { i.e.,} $v_p(y^2-1)<-v_p(s)$}. By the choice of $y$, there exist $x,z,w\in \Z$ such that $u^{2} x w- M z y=1$, i.e., $\psmat{u x}{y}{M z}{u w}, \psmat{ u x}{1}{M z y}{u w}\in \tilde{\Gamma}$. Since $L_{s,\frac{g}{t}}$ fixed by the matrices $\psmat{u x}{y}{M z}{u w}, \psmat{u x}{1}{M z y}{u w}$, we must have $s(u xe_1+ye_2)+\frac{g}{t}(M z e_1+u w e_2), s(u xe_1+e_2)+\frac{g}{t}(M z y e_1+u w e_2)\in L_{s,\frac{g}{t}}$. Thus there exist $c_1,c_2,d_1,d_2\in \Z$ such that
	\begin{equation}
		s(u xe_1+ye_2)+\frac{g}{t}(M z e_1+u w e_2)=c_1(s e_1+\frac{g}{t}e_2)+c_2 e_2, \ \mathrm{and}
	\end{equation}
	\begin{equation}
		s(u x e_1+e_2)+\frac{g}{t}(M z y e_1+u w e_2)= d_1(s e_1+\frac{g}{t}e_2)+d_2e_2.
	\end{equation}
	Solving the previous equations we get
	\begin{equation}\label{value of c1,c2 for general level}
		c_1=u x+g N_1 z\in \Z, \ c_2=s y+\frac{g}{t}(u w-u x-gN_1z)\in \Z \ \mathrm{and}
	\end{equation}
	\begin{equation}\label{value of d1,d2 for general level}
		d_1=u x+gN_1 z y\in \Z, \ d_2=s+\frac{g}{t}(u w-u x-g N_1 z y)\in \Z.
	\end{equation}
	From \eqref{value of c1,c2 for general level} and \eqref{value of d1,d2 for general level} we get
	\begin{equation}\label{g/h multiple is and integer first contradiction}
		\frac{g}{t} u w(1-y)-\frac{g}{t} u x(1-y)-\frac{g}{t} g N_1 z(1-y^2)\in \Z.
	\end{equation}
	This is a contradiction since $v_p(\frac{g}{t} u w(1-y))\geq 0, v_p(\frac{g}{t} u x(1-y))\geq 0$ but $v_p(\frac{g}{t} g N_1 z(1-y^2))<0$ (it follows from the assumption on $y$). Therefore we must have $v_p(t)=0$ i.e., $v_p(s)=0$ which is not possible. Hence we conclude that $t=1$. Therefore the lattice $L_{s,g/t}$ is of the form $L_{s,0}$.

	Any matrix $\gamma=\psmat{ux}{yz}{M}{uw}\in \tilde{\Gamma}$ acts on $L_{s,0}$ as follows: $\gamma \cdot L_{s,0}=\langle s(uxe_1+yze_2), Me_1+uwe_2 \rangle_{Lat}.$
	If $\gamma$ fixes $L_{s,0}$, then we must have 
	\begin{equation}\label{equation 1.8}
		s(uxe_1+yze_2)=d_1se_1+d_2e_2 \ \mathrm{and}
	\end{equation}
	\begin{equation}\label{equation 1.9}
		Me_1+uwe_2=d_1^\prime se_1+d_2^\prime e_2,
	\end{equation}
	for some $d_1,d_2,d_1^\prime,d_2^\prime\in \Z$. From the above equations we have $d_1=ux$ and $d_2=syz$. Recall that $d_2,su\in \Z$. If $s\not\in\Z$ (observe that $s\in \Q$), then there exists a prime $p$ such that $v_p(s)<0$ but $v_p(u)>0$ and $v_p(yz)>0$. This contradicts that $u^2xw-Mzy=1$. Hence $s\in \Z$. On other hand from equation \eqref{equation 1.9} we have
	$M=d_1^\prime s$. Since $d_1^\prime,s\in \Z$, we must have $s|M$. This completes the proof.	
\end{proof}
\begin{cor}\label{normalizer as a subgroup for one square Atkin-Lehner}
    Let $u,M,N$ be as in Lemma \ref{fixed lattices of 1 Atkin-Lehner involution}. The normalizer of $\Gamma:=\langle \Gamma_0(N),w_{u^2,N}\rangle$ is a subgroup of $\Gamma_0^*(M)$ conjugated by the matrix $\Upsilon_u^{-1}$.
\end{cor}
\begin{proof}
The conjugation of $\Gamma$ by $\Upsilon_u$ is $\tilde{\Gamma}:=\tilde{\Gamma}^u_{M u}$. Recall that by Lemma \ref{fixed lattices of 1 Atkin-Lehner involution}, the lattices fixed under $\tilde{\Gamma}$ forms a $(M|1)$-snake. For $X\in (M|1)$-snake and $\sigma\in N(\tilde{\Gamma})$, we have
$$\sigma^{-1} \tilde{\Gamma} \sigma(X)=X, \; i.e.,\;\; \tilde{\Gamma} \sigma(X)=\sigma(X).$$
Thus $\tilde{\Gamma}$ fixes $\sigma(X)$, consequently $\sigma(X)\in (M|1)$-snake. Therefore we obtain that $\sigma$ set-wise fixes the $(M|1)$-snake. {By Theorem \ref{Conway}, we conclude that the normalizer of $\tilde{\Gamma}$ is contained in the group $\Gamma_0(M/1|1)+=\Gamma_0^*(M)$}. The result follows.
\end{proof}

Let us study the general situation. For certain Atkin-Lehner involutions $w_{d_1},\ldots, w_{d_n}$, we write $|\langle w_{d_1}, \ldots, w_{d_n}\rangle|$ for $|\langle \Gamma_0(N), w_{d_1}, \ldots, w_{d_n}\rangle/\Gamma_0(N)|$.

\begin{thm}\label{MainThm}
	Let $N\in \N$ such that $4,9\nmid N$ and $u_1^2,\ldots, u_k^2$ be divisors of $N$ such that $u_i^2||N$ for $i=1,\ldots, k$ and $|\langle w_{u_1^2},\ldots,w_{u_k^2}\rangle|=2^k$. Then the normalizer of $\langle \Gamma_0(N), w_{u_1^2}, \ldots, w_{u_k^2}\rangle$ is a subgroup of $\Gamma_0^*(\frac{N}{\mathrm{lcm}(u_1^2,u_2^2,\ldots, u_k^2)})$ conjugated by $\Upsilon^{-1}$, where $\Upsilon:=\Upsilon_{\mathrm{lcm}(u_1,\ldots, u_k)}$.
	
	 Moreover, if $W_N=\langle w_{u_1^2}, \ldots, w_{u_k^2}, w_{v_{k+1}},\ldots,w_{v_n}\rangle\leq B(N)$ such that $| \langle w_{v_{k+1}},\ldots,w_{v_n}\rangle|=2^{n-k}$ and for any Atkin-Lehner involution $w_d\in\langle \Gamma_0(N), w_{v_{k+1}},\ldots,w_{v_n}\rangle$, $d$ is not a perfect square, then the normalizer of $\Gamma_0(N)+W_N$ is a subgroup of $\Gamma_0^*(M)$ conjugated by $\Upsilon^{-1}$ where $M=\frac{N}{\mathrm{lcm}(u_1^2,u_2^2,\ldots, u_k^2)}$. In general, we have
	$$\mathcal{N}(\langle \Gamma_0(N),W_N\rangle)\leq \mathcal{N}(\langle \Gamma_0(N),w_{u_1^2},\ldots,w_{u_k^2}\rangle).$$
\end{thm}
\begin{proof}
    We first prove the statement regarding the normalizer of $\langle \Gamma_0(N),w_{u_1^2},\ldots,w_{u_k^2}\rangle$.
    The case $k=1$ follows from Corollary \ref{normalizer as a subgroup for one square Atkin-Lehner}. Consider the group $\langle\Gamma_0(N), w_{u_1^2}, w_{u_2^2}\rangle$.

    Taking conjugation by $\Upsilon_{u_1}$, we get the group $\langle \tilde{\Gamma}^{u_1}_{Mu_1}, \Upsilon_{u_1}w_{u_2^2}\Upsilon_{u_1}^{-1}\rangle$, where $M=\frac{N}{u_1^2}$. Recall that by Lemma \ref{fixed lattices of 1 Atkin-Lehner involution}, the lattices fixed by $\tilde{\Gamma}^{u_1}_{Mu_1}$ forms a $(M|1)$-snake.

    Now taking conjugation by $\Upsilon_{{u_2}/{(u_1,u_2)}}$, we obtain the group $\tilde{W}_{u_1,u_2}:=\langle \Upsilon_{u_2/(u_1,u_2)}\tilde{\Gamma}^{u_1}_{Mu_1}\Upsilon_{u_2/(u_1,u_2)}^{-1}, \Upsilon_{\frac{u_1u_2}{(u_1,u_2)}}w_{u_2^2}\Upsilon_{\frac{u_1u_2}{(u_1,u_2)}}^{-1}\rangle.$

    If $\tilde{W}_{u_1,u_2}$ fixes the lattice $L_{s,g/t}$, then the group $\tilde{\Gamma}^{u_1}_{Mu_1}$ fixes the lattice $$\psmat{1}{0}{0}{\frac{(u_1,u_2)}{u_2}}\cdot L_{s,g/t}=\langle se_1+\frac{(u_1,u_2)}{u_2}\frac{g}{t}e_2, \frac{(u_1,u_2)}{u_2}e_2\rangle_{Lat}
		=\langle s\frac{u_2}{(u_1,u_2)}e_1+\frac{g}{t}e_2,e_2\rangle_{Lat}
		= L_{s\frac{u_2}{(u_1,u_2)},g/t}.$$
    By Lemma \ref{fixed lattices of 1 Atkin-Lehner involution}, we have $t=1$ and $1|s\frac{u_2}{(u_1,u_2)}|M$. Hence the lattice $L_{s,g/t}$ is of the form $L_{s,0}$, where $1|s\frac{u_2}{(u_1,u_2)}|M$.

Let $x,y,w\in \Z$ such that $u_2^2xw-\frac{N}{u_2^2}y=1$, then $\psmat{u_2x}{\frac{yu_1}{(u_1,u_2)}}{\frac{N}{u_2}\frac{(u_1,u_2)}{u_1u_2}}{u_2w}\in \tilde{W}_{u_1,u_2}$.
    If $\tilde{W}_{u_1,u_2}$ fixes the lattice $L_{s,0}$, then the matrix $\psmat{u_2x}{\frac{yu_1}{(u_1,u_2)}}{\frac{N}{u_2}\frac{(u_1,u_2)}{u_1u_2}}{u_2w}$ fixes the lattice $L_{s,0}$. Hence there exist $d_1,d_2,d_3,d_4\in \Z$ such that
    \begin{equation}\label{eq 1.19 new}
        s(u_2xe_1+\frac{u_1y}{(u_1,u_2)}e_2)=d_1se_1+d_2e_2,
    \end{equation}
    \begin{equation}\label{eq 1.20 new}
        \frac{N}{u_2}\frac{(u_1,u_2)}{u_1u_2}e_1+u_2we_2=d_3se_1+d_4e_2.
    \end{equation}
    From \eqref{eq 1.19 new}, we have $d_1=u_2x$ and $d_2=s\frac{u_1y}{(u_1,u_2)}$. Recall that $s\frac{u_2}{(u_1,u_2)}\in \Z$. If $s\notin \Z$, then there exists a prime $p$ such that $v_p(s)<0$ but $v_p(\frac{u_2}{(u_1,u_2)})>0$ such that $v_p(s)+v_p(\frac{u_2}{(u_1,u_2)})\geq 0$. In particular we have $v_p(u_2)>0$. Since $s\frac{u_1}{(u_1,u_2)}y\in \Z$ and $\frac{u_1}{(u_1,u_2)}, \frac{u_2}{(u_1,u_2)}$ has no common factor, we must have $v_p(y)>0$. This contradict the assumption that $u_2^2xw-\frac{N}{u_2^2}y=1$. Hence $s\in \Z$.

    From \eqref{eq 1.20 new}, we have $d_3s=\frac{N}{u_2}\frac{(u_1,u_2)}{u_1u_2}$, i.e., $s|\frac{N(u_1,u_2)}{u_1u_2^2}$. On the other hand we also have $s|\frac{N(u_1,u_2)}{u_1^2u_2}$. Since $\big(\frac{N(u_1,u_2)}{u_1u_2^2}, \frac{N(u_1,u_2)}{u_1^2,u_2} \big)=\frac{N}{\mathrm{lcm}(u_1^2,u_2^2)}$, we conclude that $s|\frac{N}{\mathrm{lcm}(u_1^2,u_2^2)}$. This completes the proof for the case $k=2$.
    
    For the general case assume that { $w_{u_t^2}\notin \langle \Gamma_0(N), w_{u_1^2}, \ldots, w_{u_{t-1}^2}\rangle$}. Now the result follows by applying induction on { $\Upsilon_{\frac{u_t \mathrm{lcm}(u_1, \ldots, u_{t-1})}{(u_t,\mathrm{lcm}(u_1,\ldots, u_{t-1}))}}\langle \Gamma_0(N), w_{u_1^2}, \ldots, w_{u_{t-1}^2},w_{u_{t}^2}\rangle \Upsilon_{\frac{u_t \mathrm{lcm}(u_1, \ldots, u_{t-1})}{(u_t,\mathrm{lcm}(u_1,\ldots, u_{t-1}))}}^{-1}$}	and proceeding similarly as in the case $k=2$.
	
	
	We now prove the statement regarding the normalizer of $\langle \Gamma_0(N),W_N\rangle$, which is inspired from ideas of \cite{Lang}, and follows from the arguments introduced in \S 1.2.    
	
	Consider the group $\tilde{\Gamma}:=\Upsilon\langle \Gamma_0(N),w_{u_1^2},\ldots,w_{u_k^2}\rangle \Upsilon^{-1}$, and denote by $\delta_{v_t}=\Upsilon w_{v_t,N}\Upsilon^{-1}$ for $t\in\{k+1,\ldots,n\}$, and write $M=N/\mathrm{lcm}(u_1^2,\ldots,u_k^2)$. Observe that the assumptions on $W_N$ imply $\delta_{v_t}\notin \mathrm{PGL}_2(\Q)$ and $\delta_{v_{t1}}^{-1}\delta_{v_{t2}}\notin \mathrm{PGL}_2(\Q)$ for $t1\neq t2$, and $\delta_{v_t}^2\in\tilde{\Gamma}$.

	Now consider the action of $\langle \tilde{\Gamma},\delta_{v_{k+1}},\ldots, \delta_{v_n}\rangle$ on the Big Picture. We use the following notations:
	\begin{itemize}
		\item $\tilde{t}: \ \langle \tilde{\Gamma},\delta_{v_{k+1}},\ldots,  \delta_{v_n}\rangle$ orbit of size $2^{n-k}$, $\bullet$ $\tilde{T}_N:$ The set of all such $\tilde{t}$, orbits of size $2^{n-k}$,
		\item $\tilde{\mathcal{C}}_{W_N}$ denotes the representatives of distinct left cosets of $\tilde{\Gamma}$ in $\langle \tilde{\Gamma},\delta_{v_{k+1}},\ldots, \delta_{v_n}\rangle$.
	\end{itemize}
  Following the argument described in \S \ref{subsection 2.2} we obtain:
  \begin{itemize}
  	\item For $X\in (M|1)$-snake, 
  	$\{\sigma(X): \sigma\in \langle \tilde{\Gamma},\delta_{v_{k+1}},\ldots, \delta_{v_n}\rangle\}=\{X,\delta(X): \delta \in \tilde{\mathcal{C}}_{W_N}\setminus\{\mathrm{id}\}\}$ is a member of $\tilde{T}_N$. 
  	\item If $\tilde{t}\in\tilde{T}_N$, then $\tilde{t}$ is a subset of the $(M|1)$-snake.
  \end{itemize}
 
 Using these properties and arguing similarly as in the proofs of Lemma  \ref{number of elements in a +we1,..., wen orbit} and Lemma \ref{orbit is a subset of a snake}, we obtain $X\in (M|1)$-snake if and only if $X\in \tilde{t}\in \tilde{T}_N$. 
 

Now for each $\sigma\in \mathcal{N}(\langle \tilde{\Gamma},\delta_{v_{k+1}},\ldots, \delta_{v_n}\rangle)$, and $\tilde{t}\in \tilde{T}_N$, we have
$$(\langle \tilde{\Gamma},\delta_{v_{k+1}},\ldots, \delta_{v_n}\rangle)\sigma(\tilde{t})=\sigma(\langle \tilde{\Gamma},\delta_{v_{k+1}},\ldots, \delta_{v_n}\rangle)(\tilde{t})=\sigma(\tilde{t}).$$
Therefore $(M|1)$-snake is fixed by $\sigma$. Since $4,9 \nmid N$, using Theorem \ref{Conway} we conclude that $\sigma\in \Gamma_0(M|1)+$. This proves the second statement.
	
	For the last statement, write $\Gamma_1:=\langle\Gamma_0(N),w_{u_1^2},\ldots,w_{u_k^2} \rangle$ and $\Gamma_2:=\langle\Gamma_1,w_{v_{k+1}},\cdots, w_{v_n}\rangle$.

{ 	
	Any element of $\Gamma_2$ can be written in the form $w_{u^2}^{n_0}\prod_{i=k+1}^n w_{v_i}^{n_i}\gamma$, where $n_0,n_i\in \{0,1\}$, $u^2||\mathrm{lcm}(u_1^2,\ldots, u_k^2), w_{u^2}\in \Gamma_1$ and $\gamma\in \Gamma_0(N)$. 
	Any element of $\Gamma_1$ is of the form $w_{{u^\prime}^2}^{m_0}\gamma^\prime$, where $m_0\in \{0,1\}$, ${u^\prime}^2||\mathrm{lcm}(u_1^2,\ldots, u_k^2), w_{{u^\prime}^2}\in \Gamma_1$ and $\gamma^\prime \in \Gamma_0(N)$. Note that $w_{u^2}, \gamma\in \mathrm{PSL}_2(\Q)$ for $\gamma\in \Gamma_0(N)$ and $u^2||\mathrm{lcm}(u_1^2,\ldots, u_k^2)$. By the assumptions on $v_i$'s, we have $\prod_{i=k+1}^n w_{v_i}^{n_i}\not\in \mathrm{PSL}_2(\Q)$ for $n_i\in \{0,1\}$ with at least one of  $n_i$'s is { non-zero}. 
	
	Now consider $\tilde{\sigma}\in \mathcal{N}(\Gamma_2)$ and $w_{{u^\prime}^2}^{m_0}\gamma^\prime\in \Gamma_1$, where $m_0\in \{0,1\}$, ${u^\prime}^2||\mathrm{lcm}(u_1^2,\ldots, u_k^2)$ and $\gamma^\prime \in \Gamma_0(N)$. Since $\mathcal{N}(\Gamma_2)\subseteq \Upsilon^{-1}\Gamma_0^*(M)\Upsilon$, we have 
	$\tilde{\sigma}w_{{u^\prime}^2}^{m_0}\gamma^\prime\tilde{\sigma}^{-1}\in \mathrm{PSL}_2(\Q),$
	 $\tilde{\sigma}w_{{u^\prime}^2}^{m_0}\gamma^\prime\tilde{\sigma}^{-1}\in \Gamma_2$. Thus we can write
	\begin{equation}\label{eq 1.17 new}
		\tilde{\sigma}w_{{u^\prime}^2}^{m_0}\gamma^\prime\tilde{\sigma}^{-1}=w_{u^2}^{n_0}\prod_{i=k+1}^n w_{v_i}^{n_i}\gamma \in \Gamma_2, \ \mathrm{where} \ n_0,n_i\in \{0,1\},  \ w_{u^2}\in \Gamma_1 \ \mathrm{and} \ \gamma\in \Gamma_0(N).
	\end{equation}
	If $n_i>0$ for some $i$, then { from} \eqref{eq 1.17 new}, we have $\tilde{\sigma}w_{{u^\prime}^2}^{m_0}\gamma^\prime\tilde{\sigma}^{-1}\in \mathrm{PSL}_2(\Q)$ but $w_{u^2}^{n_0}\prod_{i=k+1}^n w_{v_i}^{n_i}\gamma\not\in \mathrm{PSL}_2(\Q)$, which is a contradiction. Hence $n_i=0$ for $k+1\leq i \leq n$. Therefore we conclude that $\tilde{\sigma}w_{{u^\prime}^2}^{m_0}\gamma^\prime\tilde{\sigma}^{-1}\in \Gamma_1$, i.e. $\tilde{\sigma}\in \mathcal{N}(\Gamma_1)$. The result follows.
	}
	\end{proof}

%% file: Exact_Normalizer_Revised.tex
\section{Exact normalizer of $\langle \Gamma_0(N),W\rangle$} \label{section exact normalizer computation}
In this section we completely determine the normalizer of $\langle \Gamma_0(N),W\rangle$ where $4,9\nmid N$ and $W$ is a subgroup generated by certain Atkin-Lehner involutions. We compute it in two steps. First we compute the exact normalizer of $\langle \Gamma_0(N), w_{u_1^2}, w_{u_2^2}, \ldots, w_{u_k^2}\rangle$. Then with the help of this result we compute the exact normalizer of $\langle \Gamma_0(N), W\rangle$ for any arbitrary subgroup $W$.

Throughout the section we assume that $4,9\nmid N$. We introduce the following notations:

 For a matrix $A=\psmat{\alpha_{0,0}}{\alpha_{0,1}}{\alpha_{1,0}}{\alpha_{1,1}}$, we write $A[i,j]:=\alpha_{i,j}$ for $i,j\in \{0,1\}$.  For a prime $p$ and $n\in \N$, we use the notation $v_p(n)$ to denote the unique integer $n_p$ such that $p^{n_p}||n$.
Consider the group $\langle\Gamma_0(N), w_{u_1^2}, \ldots, w_{u_k^2}\rangle$, where $u_i^2||N$. 
Since $\mathrm{lcm}(u_1^2,\ldots, u_k^2)||N$, we have $M:=\frac{N}{\mathrm{lcm}(u_1^2,\ldots, u_k^2)}\in \Z$. { We define}
$$\Gamma_{u}:=\Big\{\psmat{ux}{y\cdot \frac{\mathrm{lcm}(u_1,\ldots, u_k)}{{u}}}{M\cdot \frac{\mathrm{lcm}(u_1,\ldots, u_k)}{{u}}\cdot z}{uw}\in \Gamma_0(M): x,y,z,w\in \Z\Big\} \ \mathrm{for\ any} \ { u||\mathrm{lcm}(u_1,\ldots, u_k)}, \ u>1, \ \mathrm{and}$$
$$\Gamma_{(u_1,\ldots, u_k)}:=\Big\{\psmat{x}{y\cdot {\mathrm{lcm}(u_1,\ldots, u_k)}}{M\cdot {\mathrm{lcm}(u_1,\ldots, u_k)}\cdot z}{w}\in \Gamma_0(M): x,y,z,w\in \Z\Big\}.$$

Observe that any element of $\langle\Gamma_0(N), w_{u_1^2}, \ldots, w_{u_k^2}\rangle\char`\\ \Gamma_0(N)$ can be written in the form $w_{{u^\prime}^2}\gamma^\prime$, for some ${u^\prime}||\mathrm{lcm}(u_1,\ldots, u_k)$ and $\gamma^\prime\in \Gamma_0(N)$.
Furthermore, for any $w_{{u^\prime}^2}\gamma^\prime\in \langle\Gamma_0(N), w_{u_1^2}, \ldots, w_{u_k^2}\rangle \char`\\ \Gamma_0(N)$ it is easy to check that $\Upsilon_{\mathrm{lcm}(u_1,\ldots, u_k)}w_{{u^\prime}^2}\Gamma_0(N)\Upsilon_{\mathrm{lcm}(u_1,\ldots, u_k)}^{-1}=\Gamma_{u^\prime}$ and $\Gamma_{(u_1,\ldots, u_k)}=\Upsilon_{\mathrm{lcm}(u_1,\ldots, u_k)}\Gamma_0(N)\Upsilon_{\mathrm{lcm}(u_1,\ldots, u_k)}^{-1}$. Therefore
$$\Upsilon_{\mathrm{lcm}(u_1,\ldots, u_k)}\langle\Gamma_0(N), w_{u_1^2}, \ldots, w_{u_k^2}\rangle\Upsilon_{\mathrm{lcm}(u_1,\ldots, u_k)}^{-1}=\langle \Gamma_{(u_1,\ldots, u_k)}, \Gamma_{u_1}, \ldots, \Gamma_{u_k}\rangle.$$
Furthermore if $\delta \in\langle \Gamma_{(u_1,\ldots, u_k)}, \Gamma_{u_1}, \ldots, \Gamma_{u_k}\rangle\char`\\ \Gamma_{(u_1,\ldots, u_k)}$, then $\delta\in \Gamma_{u^\prime}$ for some ${u^\prime}||\mathrm{lcm}(u_1,u_2, \ldots, u_k)$. In particular we have $\langle \Gamma_{(u_1,\ldots, u_k)}, \Gamma_{u^\prime}\rangle\char`\\ \Gamma_{(u_1,\ldots, u_k)}=\Gamma_{u^\prime}$.
 We mention some basic facts about $w_{u_i^2}$'s and $\Gamma_{u_i}$'s.
\begin{itemize}
\item If $v_p(\mathrm{lcm}(u_1,u_2,\ldots, u_k))>0$ for some prime $p$, then $2v_p(\mathrm{lcm}(u_1,u_2,\ldots, u_k))=v_p(N)$.
{\item For $i\ne j$, $w_{u_i^2}w_{u_j^2}\in w_{u^2}\Gamma_0(N)$ with $u||\mathrm{lcm}(u_i,u_j)$, in particular { $u:=u_iu_j/(u_1,u_2)^2$.}
\item If $v_p(u_i),v_p(u_j)>0$, then $v_p(u_i)=v_p(u_j)=v_p(\mathrm{lcm}(u_1,u_2,\ldots, u_k))$.}
\item If $\delta \in\langle \Gamma_{(u_1,\ldots, u_k)}, \Gamma_{u_1}, \ldots, \Gamma_{u_k}\rangle\char`\\ \Gamma_{(u_1,\ldots, u_k)}$, then $\delta\in \Gamma_{u^\prime}$ for some ${u^\prime}||\mathrm{lcm}(u_1,u_2, \ldots, u_k)$. In particular we have $\langle \Gamma_{(u_1,\ldots, u_k)}, \Gamma_{u^\prime}\rangle\char`\\ \Gamma_{(u_1,\ldots, u_k)}=\Gamma_{u^\prime}$.
\item Furthermore, if $(\langle \Gamma_{(u_1,\ldots, u_k)}, \Gamma_{u_1}, \ldots, \Gamma_{u_k}\rangle \char`\\ \Gamma_{(u_1,\ldots, u_k)})\cap \Gamma_{u^\prime}$ is non-empty and $v_p(u^\prime)>0$, then $$v_p(\mathrm{lcm}(u_1,u_2,\ldots, u_k))=v_p(u^\prime).$$ 
\end{itemize}
We recall the following result from Theorem \ref{MainThm}.
\begin{lema}\label{Main theorem for computing normaliser General case new version}
Let $N,u_i\in \N$ such that $4,9\nmid N$ and $u_i^2||N$ for $i\in \{1,\ldots,k\}$. Then $\mathcal{N}(\langle \Gamma_{(u_1,\ldots, u_k)}, \Gamma_{u_1}, \ldots, \Gamma_{u_k}\rangle)$ is a subgroup of $\Gamma_0^*(M)$ (recall that $M:=\frac{N}{\mathrm{lcm}(u_1^2,\ldots, u_k^2)}$). Suppose $\{h_1,h_2,\ldots, h_n\}$ is a complete set of coset representatives of $\Gamma_0^*(M)/\langle \Gamma_{(u_1,\ldots, u_k)}, \Gamma_{u_1}, \ldots, \Gamma_{u_k}\rangle$, and consider the set 
    $$\Delta=\{h_i: h_i\gamma h_i^{-1}\in \langle \Gamma_{(u_1,\ldots, u_k)}, \Gamma_{u_1}, \ldots, \Gamma_{u_k}\rangle \ {for \ every} \ \gamma\in \langle \Gamma_{(u_1,\ldots, u_k)}, \Gamma_{u_1}, \ldots, \Gamma_{u_k}\rangle \}.$$
    Then $\mathcal{N}(\langle \Gamma_{(u_1,\ldots, u_k)}, \Gamma_{u_1}, \ldots, \Gamma_{u_k}\rangle)$ is generated by $\{\gamma, h_i: \gamma \in \langle \Gamma_{(u_1,\ldots, u_k)}, \Gamma_{u_1}, \ldots, \Gamma_{u_k}\rangle, h_i\in \Delta \}$.
\end{lema}

Consider the set
$$S_{(u_1,\ldots,u_k),M}^\prime:=\Big\{\psmat{a}{b}{c}{d}\in \Gamma_0(M)\char`\\ \langle \Gamma_{(u_1,\ldots, u_k)}, \Gamma_{u_1}, \ldots, \Gamma_{u_k}\rangle: ac\equiv bd\equiv 0 \pmod {\mathrm{lcm}(u_1,\ldots,u_k)}\Big\}.$$
Observe that for $\psmat{a}{b}{c}{d}\in S_{(u_1,\ldots,u_k),M}^\prime$, if $u^\prime=(a,\mathrm{lcm}(u_1,\ldots,u_k))$ and $u^{\prime \prime}=(c,\mathrm{lcm}(u_1,\ldots,u_k))$, then 
$$\mathrm{lcm}(u_1,\ldots,u_k)=u^\prime u^{\prime\prime}, (b,\mathrm{lcm}(u_1,\ldots,u_k))=u^{\prime\prime} \ \mathrm{and} \ (d,\mathrm{lcm}(u_1,\ldots,u_k))=u^\prime.$$
{ Let $g:=\psmat{a}{b}{c}{d}\in S_{(u_1,\ldots,u_k),M}^\prime$ and $u^\prime=(a,\mathrm{lcm}(u_1,\ldots,u_k))$ (note that this implies $u^\prime||\mathrm{lcm}(u_1,\ldots,u_k)$ and ${u^\prime}^2||N$).} It is easy to check that $\Upsilon_{\mathrm{lcm}(u_1,\ldots, u_k)}^{-1}g \Upsilon_{\mathrm{lcm}(u_1,\ldots, u_k)}\in w_{{u^\prime}^2,N}\Gamma_0(N)$. Since $w_{{u^\prime}^2,N}\in \mathcal{N}(\langle\Gamma_0(N), w_{u_1^2}, \ldots, w_{u_k^2}\rangle)$, we have $g\in \mathcal{N}(\langle \Gamma_{(u_1,\ldots, u_k)}, \Gamma_{u_1}, \ldots, \Gamma_{u_k}\rangle)$. Therefore we obtain
\begin{lema}\label{S prime contained in the normalizer proposition}
$S_{(u_1,\ldots,u_k),M}^\prime\subseteq \mathcal{N}(\langle \Gamma_{(u_1,\ldots, u_k)}, \Gamma_{u_1}, \ldots, \Gamma_{u_k}\rangle)$.
\end{lema}

For $e||M$, we fix $x_{(e,M,k)},y_{(e,M,k)}\in \Z$ such that $ey_{(e,M,k)}-\frac{M}{e} \mathrm{lcm}(u_1,\ldots,u_k)^2x_{(e,M,k)}=1$ i.e., $$\delta_{(e,M,k)}:=\frac{1}{\sqrt{e}}\psmat{e}{\mathrm{lcm}(u_1,\ldots, u_k)\cdot x_{(e,M,k)}}{M\cdot \mathrm{lcm}(u_1,\ldots, u_k)}{e\cdot y_{(e,M,k)}}\in w_{e,M}\Gamma_0(M).$$ Note that the set $\{\mathrm{id},\delta_{(e,M,k)}: e||M\}$ forms a complete set of representatives for the left cosets of $\Gamma_0(M)$ in $\Gamma_0^*(M)$. If the set $S_{(u_1,\ldots,u_k)}^M:=\{g_i : 1\leq i \leq [\Gamma_0(M):\langle \Gamma_{(u_1,\ldots, u_k)}, \Gamma_{u_1}, \ldots, \Gamma_{u_k}\rangle]\}$ forms a complete set of representatives for the left cosets of $\langle \Gamma_{(u_1,\ldots, u_k)}, \Gamma_{u_1}, \ldots, \Gamma_{u_k}\rangle$ in $\Gamma_0(M)$, then the set 
$$S_{(u_1,\ldots,u_k)}^{M,+}:= \{\delta_{(e,M,k)}^jg_i: 0\leq j \leq 1, e||M,1\leq i \leq [\Gamma_0(M):\langle \Gamma_{(u_1,\ldots, u_k)}, \Gamma_{u_1}, \ldots, \Gamma_{u_k}\rangle]\}$$ forms a complete set of representatives for the left cosets of $\langle \Gamma_{(u_1,\ldots, u_k)}, \Gamma_{u_1}, \ldots, \Gamma_{u_k}\rangle$ in $\Gamma_0^*(M)$.

Since $\Upsilon_{\mathrm{lcm}(u_1,\ldots, u_k)}^{-1}\delta_{(e,M,k)} \Upsilon_{\mathrm{lcm}(u_1,\ldots, u_k)}\in w_{e,N}\Gamma_0(N)$ (with $e||M$) and $w_{e,N}\in \mathcal{N}(\langle\Gamma_0(N), w_{u_1^2}, \ldots, w_{u_k^2}\rangle),$ we have $\delta_{(e,M,k)}\in \mathcal{N}(\langle \Gamma_{(u_1,\ldots, u_k)}, \Gamma_{u_1}, \ldots, \Gamma_{u_k}\rangle)$. Therefore by { Lemma \ref{Main theorem for computing normaliser General case new version}} it suffices to compute the $g_i$'s such that { $g_i\in \mathcal{N}(\langle \Gamma_{(u_1,\ldots, u_k)}, \Gamma_{u_1}, \ldots, \Gamma_{u_k}\rangle)$}, i.e., we need to compute the set $\mathcal{N}(\langle \Gamma_{(u_1,\ldots, u_k)}, \Gamma_{u_1}, \ldots, \Gamma_{u_k}\rangle)\cap \Gamma_0(M)$. 

\begin{prop}\label{ac and bd divisible by lcm/5 Theorem}
Let $N, u_1,u_2\ldots, u_k\in \N$ such that $4,9\nmid N$ and $u_i^2||N$ for $i\in \{1,2,\ldots,k\}$. 
If $$\psmat{a}{b}{c}{d}\in \mathcal{N}(\langle \Gamma_{(u_1,\ldots, u_k)}, \Gamma_{u_1}, \ldots, \Gamma_{u_k}\rangle)\cap \Gamma_0(M), \  \mathrm{then} \ 
ac\equiv bd \equiv 0 \pmod {\frac{\mathrm{lcm}(u_1,\ldots, u_k)}{5^{v_5(\mathrm{lcm}(u_1,\ldots, u_k))}}}.$$
\end{prop}
\begin{proof}
{ Let $\sigma:=\psmat{a}{b}{c}{d} \in \mathcal{N}(\langle \Gamma_{(u_1,\ldots, u_k)}, \Gamma_{u_1}, \ldots, \Gamma_{u_k}\rangle)\cap \Gamma_0(M)$
%
and $p_1>5$ be a prime such that $p_1|\mathrm{lcm}(u_1,u_2,\ldots, u_k)$.} Without loss of generality we assume that $v_{p_1}(u_1)=v_{p_1}(\mathrm{lcm}(u_1,u_2,\ldots, u_k))$, and write $n_1:=v_{p_1}(u_1)$.

For $l\in \{1,2,3\}$, there exist $r_{u_1,l}, k_{u_1,l}\in \Z$ such that
\begin{equation}\label{lcm Atkin Lehner u1 1}
u_1^2k_{u_1,l}+lM\frac{\mathrm{lcm}(u_1,\ldots, u_k)^2}{u_1^2}r_{u_1,l}=1,
\end{equation}
i.e., $\gamma_{u_1,l}:=\psmat{u_1}{-l\frac{\mathrm{lcm}(u_1,\ldots, u_k)}{u_1}}{M\frac{\mathrm{lcm}(u_1,\ldots, u_k)}{u_1}r_{u_1,l}}{u_1k_{u_1,l}}\in \langle \Gamma_{(u_1,\ldots, u_k)}, \Gamma_{u_1}\rangle \subseteq \langle \Gamma_{(u_1,\ldots, u_k)}, \Gamma_{u_1}, \ldots, \Gamma_{u_k}\rangle.$  Consequently, $\sigma \gamma_{u_1,l}\sigma^{-1}\in \langle \Gamma_{(u_1,\ldots, u_k)}, \Gamma_{u_1}, \ldots, \Gamma_{u_k}\rangle$. In particular, $\sigma \gamma_{u_1,l}\sigma^{-1}\in \langle \Gamma_{(u_1,\ldots, u_k)}, \Gamma_{v_l}\rangle$ for some $v_l||\mathrm{lcm}(u_1,u_2, \ldots, u_k)$.

Suppose there exist $l_1,l_2\in \{1,2,3\}$ (with $l_1\neq l_2$) such that $\sigma \gamma_{u_1,l_1}\sigma^{-1}\in \langle \Gamma_{(u_1,\ldots, u_k)}, \Gamma_{u_1^\prime}\rangle$ and $\sigma \gamma_{u_1,l_2}\sigma^{-1}\in \langle \Gamma_{(u_1,\ldots, u_k)}, \Gamma_{u_2^{\prime}}\rangle$ for some $u_1^\prime, u_2^{\prime}$ with $p_1\nmid u_1^\prime u_2^\prime$.
{Then from the constructions of $\Gamma_{(u_1,\ldots, u_k)}, \Gamma_{u_i^{\prime}}$ we have (see the discussions before Lemma \ref{Main theorem for computing normaliser General case new version})}
\begin{equation}\label{conjugate in u1prime 1st case}
\sigma \gamma_{u_1,l_i}\sigma^{-1}[0,1]\equiv \sigma \gamma_{u_1,l_i}\sigma^{-1}[1,0] \equiv 0 \pmod {\frac{\mathrm{lcm}(u_1,\ldots, u_k)}{u_i^\prime}}, \  \mathrm{for} \ i \in \{1,2\}.
\end{equation}
Since $p_1\nmid u_1^\prime u_2^\prime$, we have $p_1^{n_1}|\frac{\mathrm{lcm}(u_1,\ldots, u_k)}{u_1^\prime}$ and $p_1^{n_1}|\frac{\mathrm{lcm}(u_1,\ldots, u_k)}{u_2^\prime}$. From \eqref{conjugate in u1prime 1st case} 
we have
\begin{equation}\label{conjugate in u1prime 1st case 2nd equation}
\sigma \gamma_{u_1,l_i}\sigma^{-1}[0,1]\equiv \sigma \gamma_{u_1,l_i}\sigma^{-1}[1,0]  \equiv 0\pmod {p_1^{n_1}}, \ \mathrm{for} \ i \in \{1,2\}.
\end{equation}
Combining \eqref{lcm Atkin Lehner u1 1} with \eqref{conjugate in u1prime 1st case 2nd equation},
 we get
\begin{equation}\label{eq3.5}
a^2\frac{\mathrm{lcm}(u_1,\ldots, u_k)^2}{u_1^2}l_i^2 + b^2\equiv c^2\frac{\mathrm{lcm}(u_1,\ldots, u_k)^2}{u_1^2}l_i^2 + d^2 \equiv 0 \pmod {p_1^{n_1}}, \ \mathrm{for} \ i \in \{1,2\}.
\end{equation}
From \eqref{eq3.5} we obtain
\begin{equation}\label{conjugate in u1prime, u2prime final equation 1}
a^2\frac{\mathrm{lcm}(u_1,\ldots, u_k)^2}{u_1^2}(l_1^2-l_2^2) \equiv c^2\frac{\mathrm{lcm}(u_1,\ldots, u_k)^2}{u_1^2}(l_1^2-l_2^2) \equiv 0 \pmod {p_1^{n_1}}.
\end{equation}
Recall that $p\nmid (l_1^2-l_2^2)$ for any prime $p>5$ and $u_1||\mathrm{lcm}\{u_1,\ldots,u_k\}$.
Since $p_1^{n_1}>5$ and $p_1\nmid \frac{\mathrm{lcm}(u_1,\ldots, u_k)^2}{u_1^2}$, \eqref{conjugate in u1prime, u2prime final equation 1} implies that $(a,c)>1$. Which contradicts that $\psmat{a}{b}{c}{d}\in \Gamma_0(M)$. 
{Therefore, for any two distinct elements $i_1,i_2\in \{1,2,3\}$ we must have $\{\sigma \gamma_{u_1,i_1}\sigma^{-1}, \sigma \gamma_{u_1,i_2}\sigma^{-1}\}\not\subset \Gamma_{(u_1,\ldots, u_k)}$, $\sigma \gamma_{u_1,i_1}\sigma^{-1}\in \langle \Gamma_{(u_1,\ldots, u_k)}, \Gamma_{v_1^\prime}\rangle$ and $\sigma \gamma_{u_1,i_2}\sigma^{-1}\in \langle \Gamma_{(u_1,\ldots, u_k)}, \Gamma_{v_2^{\prime}}\rangle$ for some $v_1^\prime, v_2^{\prime}$ with $p_1| v_1^\prime v_2^\prime$ i.e., either $p|v_1$ or $p|v_2$. 
Hence there exist $l_1,l_2\in \{1,2,3\}$ (with $l_1\neq l_2$) such that $\sigma \gamma_{u_1,l_1}\sigma^{-1}\in \Gamma_{u_3^\prime}$ and $\sigma \gamma_{u_1,l_2}\sigma^{-1}\in  \Gamma_{u_4^{\prime}}$ for some $u_3^\prime, u_4^{\prime}$ with $p_1| (u_3^\prime, u_4^\prime)$ (for example suppose that $i_1\in \{1,2,3\}$ such that $\sigma \gamma_{u_1,i_1}\sigma^{-1}\in \langle \Gamma_{(u_1,\ldots, u_k)}, \Gamma_{v_1^\prime}\rangle$ with $p\nmid v_1^\prime$, then for the two remaining elements $i_2,i_3\in \{1,2,3\}\char`\\ \{i_1\}$ we must have $\sigma \gamma_{u_1,i_2}\sigma^{-1}\in \Gamma_{v_2^\prime}$ and $\sigma \gamma_{u_1,i_3}\sigma^{-1}\in  \Gamma_{v_3^{\prime}}$ for some $v_2^\prime, v_3^{\prime}$ with $p_1| v_2^\prime$ and $p|v_3^\prime$).}
 Therefore we have
\begin{equation}\label{conjugate in u3prime 1}
\sigma \gamma_{u_1,l_i}\sigma^{-1}[0,0]\equiv \sigma \gamma_{u_1,l_i}\sigma^{-1}[1,1] \equiv 0 \pmod {u_{2+i}^\prime}, \ \mathrm{for} \ i\in \{1,2\},
\end{equation}
\begin{equation}\label{conjugate in u3prime inverse 1}
\sigma \gamma_{u_1,l_i}\sigma^{-1}[0,1]\equiv \sigma \gamma_{u_1,l_i}\sigma^{-1}[1,0] \equiv 0 \pmod {\frac{\mathrm{lcm}(u_1,\ldots, u_k)}{u_{2+i}^\prime}}, \ \mathrm{for} \ i\in \{1,2\}.
\end{equation}
Recall that from \eqref{lcm Atkin Lehner u1 1} we have
\begin{equation}
l_iM\frac{\mathrm{lcm}(u_1,\ldots, u_k)^2}{u_1^2}r_{u_1,l_i}\equiv 1 \pmod {p_1^{n_1}} \ \mathrm{for } \ i\in \{1,2\}.
\end{equation}
Using this congruence, from \eqref{conjugate in u3prime 1} we have
\begin{equation}\label{ac+bd with l1 and l2 mod eta1}
ac\frac{\mathrm{lcm}(u_1,\ldots, u_k)^2}{u_1^2}l_1^2 + bd\equiv ac\frac{\mathrm{lcm}(u_1,\ldots, u_k)^2}{u_1^2}l_2^2 + bd\equiv 0\pmod {p_1^{n_1}}.
\end{equation}
Thus we obtain
\begin{equation}
ac\frac{\mathrm{lcm}(u_1,\ldots, u_k)^2}{u_1^2}(l_1^2-l_2^2) \equiv 0 \pmod {p_1^{n_1}},
\end{equation}
equivalently we get
\begin{equation}\label{ac l1-l2 equiv 0 mod eta1 1}
ac(l_1^2-l_2^2) \equiv 0 \pmod {p_1^{n_1}}. 
\end{equation}
Since $(p_1^{n_1}, |l_1^2-l_2^2|)=1$, \eqref{ac l1-l2 equiv 0 mod eta1 1} implies that $ac \equiv 0 \pmod {p_1^{n_1}}$. Since $p_1$ is arbitrary, we conclude that $ac\equiv bd \equiv 0 \pmod {p^{v_p(\mathrm{lcm}(u_1,u_2,\ldots, u_k))}}$ for every prime $p>5$. The result follows.
\end{proof}
In order to compute the set $\mathcal{N}(\langle \Gamma_{(u_1,\ldots, u_k)}, \Gamma_{u_1}, \ldots, \Gamma_{u_k})\rangle\cap \Gamma_0(M)$ explicitly,  first we consider the case $w_{25}\notin \langle \Gamma_0(N), w_{u_1^2}, w_{u_2^2}, \ldots, w_{u_k^2}\rangle$ and then we consider the case $w_{25}\in \langle \Gamma_0(N), w_{u_1^2}, w_{u_2^2}, \ldots, w_{u_k^2}\rangle$.
\subsection{Exact normalizer of $\langle \Gamma_0(N), w_{u_1^2}, w_{u_2^2}, \ldots, w_{u_k^2}\rangle$ with $w_{25}\notin \langle \Gamma_0(N), w_{u_1^2}, w_{u_2^2}, \ldots, w_{u_k^2}\rangle$}
The following result will be very useful for computing the normalizer when $5|u_i$.
\begin{lema}\label{5 adic valuation of abcd proposition}
Let $N, u_1,u_2\ldots, u_k\in \N$ such that $4,9\nmid N, u_i^2||N$ and $w_{5^2}\notin \langle\Gamma_0(N), w_{u_1^2}, \ldots, w_{u_k^2}\rangle$. If $$\psmat{a}{b}{c}{d}\in \mathcal{N}(\langle \Gamma_{(u_1,\ldots, u_k)}, \Gamma_{u_1}, \ldots, \Gamma_{u_k}\rangle)\cap \Gamma_0(M), \ then \ v_5(abcd)\geq 2v_5(\mathrm{lcm}(u_1,u_2,\ldots,u_k)).$$
\end{lema}
\begin{proof}
For simplicity of notations, we write $\eta:=\mathrm{lcm}(u_1,u_2,\ldots,u_k)$ and $n_0:=v_5(\eta)$. If $n_0=0$, then the proposition is obvious, so we assume that $n_0>0$.
Since $w_{5^2}\notin \langle\Gamma_0(N), w_{u_1^2}, \ldots, w_{u_k^2}\rangle$, the set $\Gamma_5\cap \langle \Gamma_{(u_1,\ldots, u_k)}, \Gamma_{u_1}, \ldots, \Gamma_{u_k}\rangle$ is empty.
 Let $\sigma:=\psmat{a}{b}{c}{d}\in  \mathcal{N}(\langle \Gamma_{(u_1,\ldots, u_k)}, \Gamma_{u_1}, \ldots, \Gamma_{u_k}\rangle)\cap \Gamma_0(M)$. 
Since $(2,\eta)=1$, there exists a prime $p$ such that $p\equiv 2 \pmod \eta$ and $p\nmid N$. Since $u>5$ for any $\Gamma_u\subseteq \langle \Gamma_{(u_1,\ldots, u_k)}, \Gamma_{u_1}, \ldots, \Gamma_{u_k}\rangle \char`\\ \Gamma_{(u_1,\ldots, u_k)}$, we have $p^2\equiv 4\not\equiv \pm 1 \pmod u$ (because $u|\eta$ and $p^2\equiv 4\not\equiv \pm 1 \pmod \eta$). Moreover, there exist $k^\prime,r\in \Z$ such that $pk^\prime-\eta^2Mr=1$ i.e., $\psmat{p}{\eta}{\eta Mr}{k^\prime}\in \Gamma_{(u_1,\ldots, u_k)}$. Consequently, we have $pk^\prime\equiv 2k^\prime\equiv 1 \pmod \eta$. 

Since $\sigma\in \mathcal{N}(\langle \Gamma_{(u_1,\ldots, u_k)}, \Gamma_{u_1}, \ldots, \Gamma_{u_k}\rangle)$, we have $E:= \sigma \psmat{p}{\eta}{\eta Mr}{k^\prime} \sigma^{-1}\in \langle \Gamma_{(u_1,\ldots, u_k)}, \Gamma_{u_1}, \ldots, \Gamma_{u_k}\rangle$. Therefore $E\in \langle \Gamma_{(u_1,\ldots, u_k)}, \Gamma_{u}\rangle$ for some $\Gamma_u\subseteq \langle \Gamma_{(u_1,\ldots, u_k)}, \Gamma_{u_1}, \ldots, \Gamma_{u_k}\rangle \char`\\ \Gamma_{(u_1,\ldots, u_k)}$. 
Suppose $E\in \Gamma_u$ for some $\Gamma_u\in \langle \Gamma_{(u_1,\ldots, u_k)}, \Gamma_{u_1}, \ldots, \Gamma_{u_k}\rangle \char`\\ \Gamma_{(u_1,\ldots, u_k)}$. Then $E[0,0]\equiv E[1,1]\equiv 0 \pmod u$. Thus 
\begin{equation}
E[0,0]+E[1,1]\equiv (p+k^\prime)(ad-bc)\equiv 2+k^\prime \equiv 0 \pmod u.
\end{equation}
The congruences $2k^\prime\equiv 1 \pmod u$ and $2+k^\prime \equiv 0 \pmod u$, imply that $2^2\equiv -1 \pmod u$, which is not possible. Therefore $E\not\in \Gamma_u$ for any $\Gamma_u\subseteq \langle \Gamma_{(u_1,\ldots, u_k)}, \Gamma_{u_1}, \ldots, \Gamma_{u_k}\rangle \char`\\ \Gamma_{(u_1,\ldots, u_k)}$.

Now suppose that $E\in \Gamma_{(u_1,\ldots, u_k)}$. Then $E[1,0]\equiv E[0,1]\equiv 0 \pmod \eta$. Consequently, we have
\begin{equation}
E[1,0]\cdot E[0,1] \equiv (p-k^\prime)^2abcd  \equiv 0 \pmod {\eta^2}.
\end{equation}
Thus $v_5((p-k^\prime)^2abcd)\geq 2v_5(\eta)=2n_0$. Since $n_0>0$, we have $p\equiv 2 \pmod 5$ and $pk^\prime\equiv 2k^\prime\equiv 1 \pmod 5$. If $v_5(p-k^\prime)>0$, then the congruence $pk^\prime\equiv 2k^\prime\equiv 1 \pmod 5$ implies $4\equiv 1 \pmod 5$, which is not possible. Hence $v_5(p-k^\prime)=0$. Consequently, we get $v_5(abcd)\geq 2n_0$.
\end{proof}
Now we compute the set $\mathcal{N}(\langle \Gamma_{(u_1,\ldots, u_k)}, \Gamma_{u_1}, \ldots, \Gamma_{u_k})\rangle\cap \Gamma_0(M)$ when $w_{5^2}\notin \langle\Gamma_0(N), w_{u_1^2}, \ldots, w_{u_k^2}\rangle$.
\begin{prop}\label{Normalizer with 25 not in quotient}
Let $N, u_1,u_2\ldots, u_k\in \N$ such that $4,9\nmid N, u_i^2||N$ and $w_{5^2}\notin \langle\Gamma_0(N), w_{u_1^2}, \ldots, w_{u_k^2}\rangle$. 
If $$\psmat{a}{b}{c}{d}\in \mathcal{N}(\langle \Gamma_{(u_1,\ldots, u_k)}, \Gamma_{u_1}, \ldots, \Gamma_{u_k}\rangle\cap \Gamma_0(M), \ then \ ac\equiv bd \equiv 0 \pmod {\mathrm{lcm}(u_1,\ldots, u_k)}.$$
\end{prop}
\begin{proof}
Let $\sigma:=\psmat{a}{b}{c}{d}\in \mathcal{N}(\langle \Gamma_{(u_1,\ldots, u_k)}, \Gamma_{u_1}, \ldots, \Gamma_{u_k}\rangle\cap\Gamma_0(M)$ and
 $n_0:=v_5(\mathrm{lcm}(u_1,u_2,\ldots,u_k))$. By Proposition \ref{ac and bd divisible by lcm/5 Theorem}, we know that $ac\equiv bd \equiv 0 \pmod{\frac{\mathrm{lcm}(u_1,u_2,\ldots,u_k)}{5^{n_0}}}$.
We now prove that $ac\equiv bd \equiv 0 \pmod {5^{n_0}}$. If $n_0=0$, then this is clear. Hence we assume that $n_0\geq 1$.
Recall that the assumption $w_{5^2}\notin \langle\Gamma_0(N), w_{u_1^2}, \ldots, w_{u_k^2}\rangle$ implies the set $\Gamma_5\cap \langle \Gamma_{(u_1,\ldots, u_k)}, \Gamma_{u_1}, \ldots, \Gamma_{u_k}\rangle$ is empty.

 Without loss of generality we assume that $v_5(u_1)=n_0$. 
If $n_0=1$, then there exists a prime $p_1$ ($\neq 5$) such that $n_1:=v_{p_1}(u_1)>0$ (this is possible since $w_{25}\notin \langle\Gamma_0(N), w_{u_1^2}, \ldots, w_{u_k^2}\rangle$).
We define $\eta_0:=\begin{cases}
p_1^{n_1}, \ \mathrm{if} \ n_0=1\\
5^{n_0}, \ \mathrm{otherwise}
\end{cases},$
{ $\eta_1:=\begin{cases}
5^{n_0}p_1^{n_1}, \ \mathrm{if} \ n_0=1\\
5^{n_0}, \ \mathrm{otherwise}
\end{cases}$ and $\eta_2:=\begin{cases}
p_1, \ \mathrm{if} \ n_0=1\\
5, \ \mathrm{otherwise}
\end{cases}.$
Then $\eta_1||u_1$.} Recall that for any prime $p$ if $v_p(u^\prime)>0$ for some $\Gamma_{u^\prime}\subseteq \langle \Gamma_{(u_1,\ldots, u_k)}, \Gamma_{u_1}, \ldots, \Gamma_{u_k}\rangle \char`\\ \Gamma_{(u_1,\ldots, u_k)}$, then $v_p(\mathrm{lcm}(u_1,u_2,\ldots, u_k))=v_p(u^\prime).$ 

 
 
For $l\in \{1,2,3\}$, there exist $r_{u_1,l}, k_{u_1,l}\in \Z$ such that
\begin{equation}\label{lcm Atkin Lehner u1 2}
u_1^2k_{u_1,l}+lM\frac{\mathrm{lcm}(u_1,\ldots, u_k)^2}{u_1^2}r_{u_1,l}=1,
\end{equation}
i.e., $\gamma_{u_1,l}:=\psmat{u_1}{-l\frac{\mathrm{lcm}(u_1,\ldots, u_k)}{u_1}}{M\frac{\mathrm{lcm}(u_1,\ldots, u_k)}{u_1}r_{u_1,l}}{u_1k_{u_1,l}}\in \langle \Gamma_{(u_1,\ldots, u_k)}, \Gamma_{u_1}\rangle \subseteq \langle \Gamma_{(u_1,\ldots, u_k)}, \Gamma_{u_1}, \ldots, \Gamma_{u_k}\rangle.$ 
%
%

Suppose there exist $l_1,l_2\in \{1,2,3\}$ (with $l_1\neq l_2$) such that $\sigma \gamma_{u_1,l_1}\sigma^{-1}\in \langle \Gamma_{(u_1,\ldots, u_k)}, \Gamma_{u_1^\prime}\rangle$ and $\sigma \gamma_{u_1,l_2}\sigma^{-1}\in \langle \Gamma_{(u_1,\ldots, u_k)}, \Gamma_{u_2^{\prime}}\rangle$ for some $u_1^\prime, u_2^{\prime}$ with $\eta_2\nmid u_1^\prime u_2^\prime$.
 Therefore we must have
\begin{equation}\label{conjugate in u1prime}
\sigma \gamma_{u_1,l_i}\sigma^{-1}[0,1]\equiv \sigma \gamma_{u_1,l_i}\sigma^{-1}[1,0] \equiv 0 \pmod {\frac{\mathrm{lcm}(u_1,\ldots, u_k)}{u_i^\prime}}, \ \ \mathrm{for} \ i \in \{1,2\}.
\end{equation}
Since $\eta_2\nmid u_1^\prime u_2^\prime$, we have $\eta_0|\frac{\mathrm{lcm}(u_1,\ldots, u_k)}{u_1^\prime}$ and $\eta_0|\frac{\mathrm{lcm}(u_1,\ldots, u_k)}{u_2^\prime}$. From \eqref{conjugate in u1prime} 
we have
\begin{equation}\label{conjugate in u1prime 2nd equation}
\sigma \gamma_{u_1,l_i}\sigma^{-1}[0,1]\equiv \sigma \gamma_{u_1,l_i}\sigma^{-1}[1,0]  \equiv 0\pmod {\eta_0}, \  \mathrm{for} \ i \in \{1,2\}.
\end{equation}
Combining \eqref{lcm Atkin Lehner u1 2} with \eqref{conjugate in u1prime 2nd equation},
we get
\begin{equation}
a^2\frac{\mathrm{lcm}(u_1,\ldots, u_k)^2}{u_1^2}l_i^2 + b^2\equiv c^2\frac{\mathrm{lcm}(u_1,\ldots, u_k)^2}{u_1^2}l_i^2 + d^2 \equiv 0 \pmod {\eta_0}, \ \mathrm{for} \ i \in \{1,2\}.
\end{equation}
Thus we have
\begin{equation}\label{conjugate in u1prime, u2prime final equation 2}
a^2\frac{\mathrm{lcm}(u_1,\ldots, u_k)^2}{u_1^2}(l_1^2-l_2^2) \equiv c^2\frac{\mathrm{lcm}(u_1,\ldots, u_k)^2}{u_1^2}(l_1^2-l_2^2) \equiv 0 \pmod {\eta_0}.
\end{equation}
Recall that $5^2\nmid (l_1^2-l_2^2)$ and $p\nmid (l_1^2-l_2^2)$ for any prime $p>5$.
Since $\eta_0^{n_1}>5$ and $\eta_0\nmid \frac{\mathrm{lcm}(u_1,\ldots, u_k)^2}{u_1^2}$, \eqref{conjugate in u1prime, u2prime final equation 2} implies that $(a,c)>1$. Which contradicts that $\psmat{a}{b}{c}{d}\in \Gamma_0(M)$. Hence there exist $l_1,l_2\in \{1,2,3\}$ (with $l_1\neq l_2$) such that $\sigma \gamma_{u_1,l_1}\sigma^{-1}\in \Gamma_{u_3^\prime}$ and $\sigma \gamma_{u_1,l_2}\sigma^{-1}\in  \Gamma_{u_4^{\prime}}$ for some $u_3^\prime, u_4^{\prime}$ with $\eta_2| (u_3^\prime, u_4^\prime)$ (which automatically implies that $\eta_0|(u_3^\prime, u_4^\prime)$) and $\Gamma_{u_3^{\prime}}, \Gamma_{u_4^{\prime}}\subseteq \langle \Gamma_{(u_1,\ldots, u_k)}, \Gamma_{u_1}, \ldots, \Gamma_{u_k}\rangle \char`\\ \Gamma_{(u_1,\ldots, u_k)}$. Therefore for $i\in \{1,2\}$ we must have
\begin{equation}\label{conjugate in u3prime second case}
\sigma \gamma_{u_1,l_i}\sigma^{-1}[0,0]\equiv \sigma \gamma_{u_1,l_i}\sigma^{-1}[1,1] \equiv 0 \pmod {u_{2+i}^\prime}, \ \mathrm{and}
\end{equation}
\begin{equation}\label{conjugate in u3prime inverse second case}
\sigma \gamma_{u_1,l_i}\sigma^{-1}[0,1]\equiv \sigma \gamma_{u_1,l_i}\sigma^{-1}[1,0] \equiv 0 \pmod {\frac{\mathrm{lcm}(u_1,\ldots, u_k)}{u_{2+i}^\prime}}.
\end{equation}
{If possible let $5\nmid (u_3^\prime, u_4^\prime)$ and WLOG assume that $5\nmid u_3^\prime$.}
Then $v_5(\frac{\mathrm{lcm}(u_1,\ldots, u_k)}{u_3^\prime})=n_0$. From \eqref{conjugate in u3prime inverse second case}, we get
\begin{equation}
a^2\frac{\mathrm{lcm}(u_1,\ldots, u_k)^2}{u_1^2}l_1^2 + b^2\equiv c^2\frac{\mathrm{lcm}(u_1,\ldots, u_k)^2}{u_1^2}l_1^2 + d^2 \equiv 0 \pmod {5^{n_0}},
\end{equation}
i.e.,
\begin{equation}
a^2d^2-b^2c^2\equiv 0 \pmod {5^{n_0}}.
\end{equation}
Since $ad-bc=1$, from the last equation we get
\begin{equation}
ad+bc\equiv 0 \pmod {5^{n_0}}.
\end{equation}
Therefore $v_5(a)=v_5(b)=v_5(c)=v_5(d)=0$,
in particular this implies $v_5(abcd)=0<n_0$, which contradicts Lemma \ref{5 adic valuation of abcd proposition}. 
Therefore $5|(u_3^\prime, u_4^\prime)$. In particular we have $5^{n_0}|(u_3^\prime, u_4^\prime)$. 
Recall that from \eqref{lcm Atkin Lehner u1 2} we have
\begin{equation}
l_iM\frac{\mathrm{lcm}(u_1,\ldots, u_k)^2}{u_1^2}r_{u_1,l_i}\equiv 1 \pmod {5^{n_0}} \ \mathrm{for } \ i\in \{1,2\}.
\end{equation}
Using this congruence, from \eqref{conjugate in u3prime second case} we have
\begin{equation}\label{ac+bd with l1 and l2 mod eta1 second case}
ac\frac{\mathrm{lcm}(u_1,\ldots, u_k)^2}{u_1^2}l_1^2 + bd\equiv ac\frac{\mathrm{lcm}(u_1,\ldots, u_k)^2}{u_1^2}l_2^2 + bd\equiv 0\pmod {5^{n_0}}.
\end{equation}
Thus we obtain
\begin{equation}
ac\frac{\mathrm{lcm}(u_1,\ldots, u_k)^2}{u_1^2}(l_1^2-l_2^2) \equiv 0 \pmod {5^{n_0}},
\end{equation}
equivalently we get
\begin{equation}\label{ac l1-l2 equiv 0 mod eta1 second case}
ac(l_1^2-l_2^2) \equiv 0 \pmod {5^{n_0}}.
\end{equation}
Since $(5^{n_0}, |l_1^2-l_2^2|)\in \{1,5\}$, \eqref{ac l1-l2 equiv 0 mod eta1 second case} implies that $5ac \equiv 0 \pmod {5^{n_0}}$.

Consider the case $(5^{n_0}, |l_1^2-l_2^2|)= 5$, $5ac \equiv 0 \pmod {5^{n_0}}$ but $ac\not\equiv 0 \pmod {5^{n_0}}$, i.e., $v_5(ac)=n_0-1$. If $v_5(bd)\ne v_5(ac\frac{\mathrm{lcm}(u_1,\ldots, u_k)^2}{u_1^2}l_1^2)$, then 
$$v_5(ac\frac{\mathrm{lcm}(u_1,\ldots, u_k)^2}{u_1^2}l_1^2 + bd)=\mathrm{min}\{v_5(bd), v_5(ac\frac{\mathrm{lcm}(u_1,\ldots, u_k)^2}{u_1^2}l_1^2)\}<n_0,$$
which contradicts \eqref{ac+bd with l1 and l2 mod eta1 second case}. Thus we have
\begin{equation}
v_5(bd)=v_5(ac\frac{\mathrm{lcm}(u_1,\ldots, u_k)^2}{u_1^2}l_1^2)=v_5(ac)=n_0-1.
\end{equation}
Consequently we get $v_5(abcd)<2n_0$, which contradicts Lemma \ref{5 adic valuation of abcd proposition}.
Hence we must have $ac\equiv 0 \pmod {5^{n_0}}$. Consequently, from \eqref{ac+bd with l1 and l2 mod eta1 second case} we obtain
$ac\equiv bd \equiv 0 \pmod {5^{n_0}}$.
 Thus, we obtain that if $\psmat{a}{b}{c}{d}\in \mathcal{N}(\langle \Gamma_{(u_1,\ldots, u_k)}, \Gamma_{u_1}, \ldots, \Gamma_{u_k}\rangle)\cap \Gamma_0(M)$, then $ac\equiv bd \equiv 0 \pmod {\mathrm{lcm}(u_1,u_2,\ldots, u_k)}$. 
%
%
%
\end{proof}
\begin{cor}\label{exact normalizer for u gt 5}
Let $N, u_1,u_2\ldots, u_k\in \N$ such that $4,9\nmid N, u_i^2||N$ and $w_{5^2}\notin \langle\Gamma_0(N), w_{u_1^2}, \ldots, w_{u_k^2}\rangle$. Then
$$\mathcal{N}(\langle \Gamma_{(u_1,\ldots, u_k)}, \Gamma_{u_1}, \ldots, \Gamma_{u_k}\rangle)=\langle \Gamma_{(u_1,\ldots, u_k)}, \Gamma_{u_1}, \ldots, \Gamma_{u_k},\delta_{(e,M,k)}, g: e||M, g\in S^\prime_{(u_1,\ldots,u_k),M}\cap S_{(u_1,\ldots,u_k)}^{M} \rangle.$$ Consequently we have
$\mathcal{N}(\langle\Gamma_0(N), w_{u_1^2}, \ldots, w_{u_k^2}\rangle)=\Gamma_0^*(N).$
\end{cor}
\begin{proof}
Recall that the set $S_{(u_1,\ldots,u_k)}^M:=\{g_i : 1\leq i \leq [\Gamma_0(M):\langle \Gamma_{(u_1,\ldots, u_k)}, \Gamma_{u_1}, \ldots, \Gamma_{u_k}\rangle]\}$ forms a complete set of representatives for the left cosets of $\Gamma_0(M)$ in $\langle \Gamma_{(u_1,\ldots, u_k)}, \Gamma_{u_1}, \ldots, \Gamma_{u_k}\rangle$, and the set 
$$S_{(u_1,\ldots,u_k)}^{M,+}:= \{\delta_{(e,M,k)}^jg_i: 0\leq j \leq 1, e||M,1\leq i \leq [\Gamma_0(M):\langle \Gamma_{(u_1,\ldots, u_k)}, \Gamma_{u_1}, \ldots, \Gamma_{u_k}\rangle]\}$$ forms a complete set of representatives for the left cosets of $\Gamma_0^*(M)$ in $\langle \Gamma_{(u_1,\ldots, u_k)}, \Gamma_{u_1}, \ldots, \Gamma_{u_k}\rangle$.

By previous discussions and Lemma \ref{S prime contained in the normalizer proposition}, we know that 
$$\mathcal{N}(\langle \Gamma_{(u_1,\ldots, u_k)}, \Gamma_{u_1}, \ldots, \Gamma_{u_k}\rangle)\supseteq \langle \Gamma_{(u_1,\ldots, u_k)}, \Gamma_{u_1}, \ldots, \Gamma_{u_k},\delta_{(e,M,k)}, g: e||M, g\in S^\prime_{(u_1,\ldots,u_k),M}\cap S_{(u_1,\ldots,u_k)}^{M} \rangle.$$
If $g_i\in \mathcal{N}(\langle \Gamma_{(u_1,\ldots, u_k)}, \Gamma_{u_1}, \ldots, \Gamma_{u_k}\rangle)\cap S_{(u_1,\ldots,u_k)}^M$, then by Proposition \ref{Normalizer with 25 not in quotient} we have $g_i\in S^\prime_{(u_1,\ldots,u_k),M}$. Since $\delta_{(e,M,k)}\in \mathcal{N}(\langle \Gamma_{(u_1,\ldots, u_k)}, \Gamma_{u_1}, \ldots, \Gamma_{u_k}\rangle)$, by Lemma \ref{Main theorem for computing normaliser General case new version} we conclude that
 $$\mathcal{N}(\langle \Gamma_{(u_1,\ldots, u_k)}, \Gamma_{u_1}, \ldots, \Gamma_{u_k}\rangle)=\langle \Gamma_{(u_1,\ldots, u_k)}, \Gamma_{u_1}, \ldots, \Gamma_{u_k},\delta_{(e,M,k)}, g: e||M, g\in S^\prime_{(u_1,\ldots,u_k),M}\cap S_{(u_1,\ldots,u_k)}^{M} \rangle.$$
 This proves the first part. For the second part, it suffices to show that 
 $$\Upsilon_{\mathrm{lcm}(u_1,\ldots, u_k)}^{-1}\langle \Gamma_{(u_1,\ldots, u_k)}, \Gamma_{u_1}, \ldots, \Gamma_{u_k},\delta_{(e,M,k)}, g: e||M, g\in S^\prime_{(u_1,\ldots,u_k),M}\cap S_{(u_1,\ldots,u_k)}^{M} \rangle\Upsilon_{\mathrm{lcm}(u_1,\ldots, u_k)}=\Gamma_0^*(N).$$
 This follows from the facts that 
 
 $\bullet$ { $\Upsilon_{\mathrm{lcm}(u_1,\ldots, u_k)}^{-1}\delta_{(e,M,k)}\Gamma_{(u_1,\ldots, u_k)} \Upsilon_{\mathrm{lcm}(u_1,\ldots, u_k)}= w_{e,N}\Gamma_0(N)$ (with $e||M$)} and
 
 $\bullet$  for any $g\in S^\prime_{(u_1,\ldots,u_k),M}$, we have $\Upsilon_{\mathrm{lcm}(u_1,\ldots, u_k)}^{-1}g \Upsilon_{\mathrm{lcm}(u_1,\ldots, u_k)}\in w_{{u^\prime}^2,N}\Gamma_0(N)$ for some $u^\prime ||\mathrm{lcm}(u_1,\ldots, u_k)$. \hspace*{0.62cm} Conversely, for any $u^\prime ||\mathrm{lcm}(u_1,\ldots, u_k)$, we have $\Upsilon_{\mathrm{lcm}(u_1,\ldots, u_k)}w_{{u^\prime}^2,N} \Upsilon_{\mathrm{lcm}(u_1,\ldots, u_k)}^{-1}\in S^\prime_{(u_1,\ldots,u_k),M}$.
\end{proof}

\subsection{Exact normalizer of $\langle \Gamma_0(N), w_{25}, w_{u_2^2}, \ldots, w_{u_k^2}\rangle$}
Now suppose that
$w_{25}\in \langle \Gamma_0(N), w_{u_1^2}, w_{u_2^2}, \ldots, w_{u_k^2}\rangle$ (note that this assumption implies $25||N$). Without loss of generality we can assume that { $u_1=5$ and $5\nmid \prod_{i=2}^k u_i$}.
We first compute the normalizer of $\langle \Gamma_0(N), w_{25}\rangle$. Then with the help of this result and Proposition \ref{ac and bd divisible by lcm/5 Theorem} we compute the normalizer of $\langle \Gamma_0(N), w_{25},  w_{u_2^2}, \ldots, w_{u_k^2}\rangle$. 

By Theorem \ref{MainThm},
we know that $\mathcal{N}(\langle \Gamma_0(N), w_{25}\rangle) \subseteq \Upsilon_{5}^{-1} \Gamma_0^*(M^\prime)\Upsilon_{5}$, where $M^\prime:=\frac{N}{25} \ (\mathrm{note \ that} \ (5,M^\prime)=1)$.
We introduce the following notations:

$$\tilde{\Gamma}_{5}(M^\prime):=\Big\{\psmat{5x}{y}{M^\prime z}{5w}\in \Gamma_0(M^\prime):x,y,z,w\in \Z \Big\},$$
$$\tilde{\Gamma}_{0}^5(5M^\prime):=\Big\{\psmat{x}{5y}{5M^\prime z}{w}\in \Gamma_0(M^\prime):x,y,z,w\in \Z \Big\}, \ \mathrm{and} \ \tilde{\Gamma}_{5M^\prime}^5:=\langle \tilde{\Gamma}_{5}(M^\prime), \tilde{\Gamma}_{0}^5(5M^\prime)\rangle.$$
For $e||M^\prime$, we fix $x_{(e,{M^\prime},1)},y_{(e,{M^\prime},1)}\in \Z$ such that $ey_{(e,{M^\prime},1)}-25\frac{{M^\prime}}{e}x_{(e,{M^\prime},1)}=1$ i.e., $$\delta_{(e,{M^\prime},1)}:=\frac{1}{\sqrt{e}}\psmat{e}{5 x_{(e,{M^\prime},1)}}{5{M^\prime}}{e\cdot y_{(e,{M^\prime},1)}}\in w_{e,{M^\prime}}\Gamma_0({M^\prime}).$$ The set $\{\mathrm{id},\delta_{(e,{M^\prime},1)}: e||{M^\prime}\}$ forms a complete set of representatives for the left cosets of $\Gamma_0^*({M^\prime})$ in $\Gamma_0({M^\prime})$.

Let $B_j:=\psmat{{M^\prime} j+1}{-j}{-{M^\prime}}{1}$ and $C_i:=\psmat{1}{i}{0}{1}$. 
First observe that $[\Gamma_0(M^\prime): \tilde{\Gamma}_0^5(5{M^\prime})]=[\Gamma_0(M^\prime): \Gamma_0(25M^\prime)]$, and the set $S_{5,{M^\prime}}:=\{B_jC_i,C_i: 0\leq i,j \leq 4\}$ forms a complete set of representatives for the left cosets of $\tilde{\Gamma}_0^5(5{M^\prime})$ in $\Gamma_0({M^\prime})$.
Let $S_{(5)}^{M^\prime}$ (a subset of $S_{5,{M^\prime}}$) be a complete set of representatives for the left cosets of $\tilde{\Gamma}_{5{M^\prime}}^5$ in $\Gamma_0({M^\prime})$.
 Consequently the set 
$S_{(5)}^{{M^\prime},+}:=\big\{\delta_{(e,{M^\prime},1)}^l g: g\in S_{(5)}^{M^\prime}, e||{M^\prime},l\in \{0,1\} \big\}$
forms a complete set of representatives for the left cosets of $\tilde{\Gamma}_{5{M^\prime}}^5$ in $\Gamma_0^*({M^\prime})$. 
\begin{prop}\label{exact normalizer for 5 before conjugation}
     We have $\mathcal{N}(\tilde{\Gamma}_{5{M^\prime}}^5)=\Big\langle \tilde{\Gamma}_{5{M^\prime}}^5, \delta_{(e,{M^\prime},1)}, B_jC_0, B_0C_i: e||{M^\prime}\Big\rangle$,
    where $0\leq j,i \leq 4$ such that  ${M^\prime}j\equiv 2\pmod 5$ and $i+j\equiv 0 \pmod 5$.
\end{prop}
\begin{proof}
It is easy to check that $\delta_{(e,{M^\prime},1)}\in \mathcal{N}(\tilde{\Gamma}_{5{M^\prime}}^5)$ and 
$C_i\notin \mathcal{N}(\tilde{\Gamma}_{5{M^\prime}}^5)$ for $i\neq 0$, hence it suffices to compute the $B_j$'s and $C_i$'s such that $B_jC_i\in S_{(5)}^{M^\prime}  \cap \mathcal{N}(\tilde{\Gamma}_{5{M^\prime}}^5)$ (cf. { Lemma \ref{Main theorem for computing normaliser General case new version}}).
We give a complete proof for the case ${M^\prime}\equiv 1 \pmod 5$, the proofs are similar for other cases. Assume that ${M^\prime}\equiv 1 \pmod 5$. 
There exist $r_1,k_1\in \Z$ such that $\gamma:=\psmat{5}{-1}{{M^\prime}r_1}{5k_1}\in \tilde{\Gamma}_5({M^\prime})$. Let $B_jC_i\notin \tilde{\Gamma}_{5{M^\prime}}^5$. If $B_jC_i\in \mathcal{N}(\tilde{\Gamma}_{5{M^\prime}}^5)$, then $E:=B_jC_i\gamma(B_jC_i)^{-1}\in \mathcal{N}(\tilde{\Gamma}_{5{M^\prime}}^5)$, i.e., either $E[0,0]\equiv E[1,1]\equiv 0 \pmod 5$ or $E[0,1]\equiv E[1,0]\equiv 0 \pmod 5$.

Suppose $E[0,0]\equiv E[1,1]\equiv 0 \pmod 5$, then
\begin{equation}\label{prop 1.10 eq 1.20}
i^2{M^\prime}({M^\prime}j+1)-i(2j{M^\prime}+1)+j+{M^\prime}({M^\prime}j+1) \equiv i^2(j+1)-i(2j+1)+2j+1\equiv 0 \pmod 5.
\end{equation}
If $(j,i)$ is a solution of \eqref{prop 1.10 eq 1.20} in $\Z/5\Z$, then it is easy to see that $(j,i)\in \{(2,0),(4,1), (1,2)\}$. Since $B_4C_1\in \tilde{\Gamma}_{5{M^\prime}}^5$, we get
$B_jC_i\in \{B_1C_2,B_2C_0\}$.
Observe that $(B_2C_0)\tilde{\Gamma}_{5{M^\prime}}^5=(B_1C_2)\tilde{\Gamma}_{5{M^\prime}}^5$.

Now suppose that $E[0,1]\equiv E[1,0]\equiv 0 \pmod 5$. The condition $E[1,0]\equiv 0 \pmod 5$ gives 
\begin{equation}\label{eq 2.20 section 2}
i^2-2i+2\equiv 0 \pmod 5.
\end{equation}
The solutions of \eqref{eq 2.20 section 2} are $i\in \{3,4\}$.
On the other hand, the relation $E[0,1]\equiv 0 \pmod 5$ gives 
\begin{equation}\label{prop 1.10 eq 1.22}
j^2(-i^2+2i-2)-j(2i^2-2i+2)-i^2-1\equiv 0 \pmod 5.
\end{equation}
For $i=3$ (resp., $i=4$), from \eqref{prop 1.10 eq 1.22} we get $j=0$ (resp., $j=3$). Therefore in this case $B_jC_i\in \{B_0C_3,B_3C_4\}$. It is easy to check that $(B_0C_3)\tilde{\Gamma}_{5{M^\prime}}^5=(B_3C_4)\tilde{\Gamma}_{5{M^\prime}}^5$.
Thus we obtain that 
$\mathcal{N}(\tilde{\Gamma}_{5{M^\prime}}^5)\subseteq \Big\langle \tilde{\Gamma}_{5{M^\prime}}^5, \delta_{(e,{M^\prime},1)}, B_2C_0, B_0C_3: e||{M^\prime}\Big\rangle.$
Now to prove the equality, it suffices to show that $B_2C_0, B_0C_3\in \mathcal{N}(\tilde{\Gamma}_{5{M^\prime}}^5)$. 

First consider a matrix of the form $\gamma:=\psmat{5x}{y}{{M^\prime}z}{5w}\in \tilde{\Gamma}_{5{M^\prime}}^5$. Since ${M^\prime}\equiv 1 \pmod 5$, we have
\begin{align*}
(B_2C_0\gamma(B_2C_0)^{-1})[0,0]&\equiv -2(y+z)\pmod 5,\ \ 
(B_2C_0\gamma(B_2C_0)^{-1})[1,1]\equiv 2(y+z)\pmod 5,\ \mathrm{and}\\
(B_2C_0\gamma(B_2C_0)^{-1})[0,1]&\equiv z-y\pmod 5,\ \ 
(B_2C_0\gamma(B_2C_0)^{-1})[1,0]\equiv z-y\pmod 5.
\end{align*}
Since $yz\equiv -1\pmod 5$, either $y+z\equiv 0\pmod 5$ or $z-y\equiv 0 \pmod 5$. Thus $B_2C_0\gamma(B_2C_0)^{-1}\in \tilde{\Gamma}_{5{M^\prime}}^5$.

Now consider the matrix of the form $\delta:=\psmat{x}{5y}{5{M^\prime}z}{w}\in \tilde{\Gamma}_{5{M^\prime}}^5$. We have
\begin{align*}
(B_2C_0\delta(B_2C_0)^{-1})[0,0]&\equiv -2(x+w)\pmod 5,\ \ 
(B_2C_0\delta(B_2C_0)^{-1})[1,1]\equiv -2(x+w)\pmod 5,\ \mathrm{and}\\
(B_2C_0\delta (B_2C_0)^{-1})[0,1]&\equiv x-w\pmod 5,\ \ 
(B_2C_0\delta(B_2C_0)^{-1})[1,0]\equiv -(x-w)\pmod 5.
\end{align*}
Since $xw\equiv 1\pmod 5$, either $x+w\equiv 0\pmod 5$ or $x-w\equiv 0 \pmod 5$. Thus $B_2C_0\delta (B_2C_0)^{-1}\in \tilde{\Gamma}_{5{M^\prime}}^5$.
Therefore we conclude that $B_2C_0\in \mathcal{N}(\tilde{\Gamma}_{5{M^\prime}}^5)$. A similar argument discussed so far shows that $B_0C_3\in \mathcal{N}(\tilde{\Gamma}_{5{M^\prime}}^5)$. Thus 
$\mathcal{N}(\tilde{\Gamma}_{5{M^\prime}}^5)\supseteq \Big\langle \tilde{\Gamma}_{5{M^\prime}}^5, \delta_{(e,{M^\prime},1)}, B_2C_0, B_0C_3: e||{M^\prime}\Big\rangle.$
This completes the proof.
\end{proof}
As an immediate consequence of Proposition  \ref{exact normalizer for 5 before conjugation}, we obtain
\begin{cor}\label{exact normalizer with only 5 square Atkin-Lehner}\label{cor3.6}
Let ${M^\prime}\in \N$ such that $(5,{M^\prime})=1$ and $4,9\nmid {M^\prime}$. Then $$\mathcal{N}(\langle \Gamma_0(5^2{M^\prime}),w_{5^2}\rangle)=\Big\langle \Gamma_0(5^2{M^\prime}),w_{5^2}, w_{e,5^2{M^\prime}}, \Upsilon_5^{-1} B_jC_0\Upsilon_5, \Upsilon_5^{-1}B_0C_i\Upsilon_5: e||{M^\prime}\Big\rangle,$$ where $0\leq j,i \leq 4$ such that  ${M^\prime}j\equiv 2\pmod 5$ and $i+j\equiv 0 \pmod 5$.
\end{cor}
We now prove that  $\mathcal{N}(\langle \Gamma_0(N), w_{25})\subseteq \mathcal{N}(\langle \Gamma_0(N), w_{25},  w_{u_2^2}, \ldots, w_{u_k^2}\rangle)$, and the following lemma plays a very important role in proving such result.

\begin{lema}\label{The order 3 element commutes with Atkin-Lehner operators}\label{lem3.9}
Let ${M^\prime}\in \N$ such that $4,9\nmid M^\prime$, $(5,{M^\prime})=1$ and $W$ be any subgroup generated by the Atkin-Lehner involutions such that $w_{5^2, 5^2M^\prime}\in W$. 
 Let $w_d:=w_{d,5^2M^\prime}\in \langle \Gamma_0(5^2{M^\prime}), W\rangle$ be an Atkin-Lehner involution and $\sigma \in \{\Upsilon_5^{-1} B_jC_0\Upsilon_5, \Upsilon_5^{-1}B_0C_i\Upsilon_5\}$, where $i,j$ are defined as in Corollary \ref{cor3.6}.
  Then $\sigma w_d\sigma^{-1}\in \langle \Gamma_0(5^2{M^\prime}), W\rangle$ if and only if $\frac{d}{(5^2,d)}\equiv \pm 1 \pmod 5$.
\end{lema}
\begin{proof} 
Note that we have $(5^2,d)\in \{1,5^2\}$. Since $\sigma w_{5^2}\sigma^{-1}\in \langle \Gamma_0(5^2{M^\prime}), w_{5^2}\rangle$, without loss of generality we can assume that $(5^2,d)=1$ (note that this also implies $d||M^\prime$).

Since $\sigma\in  \mathrm{PSL}_2(\mathbb{Z} [ \frac{1}{5} ])\char`\\  \mathrm{PSL}_2(\mathbb{Z})$, it is easy to observe that  
$\sigma w_d\sigma^{-1}\in  \mathrm{PSL}_2(\mathbb{Z}[ \frac{1}{5}, \frac{1}{\sqrt{d}}]) \char`\\  \mathrm{PSL}_2(\mathbb{Z}[ \frac{1}{5}]).$
Any arbitrary element of $\langle \Gamma_0(5^2{M^\prime}), W\rangle$ can be written in the form
$w_{d^\prime}^{m_1} w_{5^2}^{m_0}\gamma,$
where $m_0,m_1\in \{0,1\}$, $\gamma \in \Gamma_0(5^2{M^\prime})$ and $w_{d^\prime}\in W$. 

 If $\sigma w_d\sigma^{-1}\in \langle \Gamma_0(5^2{M^\prime}), W\rangle$, then we have $\sigma w_d\sigma^{-1}=w_{d^\prime}^{n}w_{5^2}^{n_0}\gamma$ i.e., $w_{d^\prime}^{-n}\sigma w_d\sigma^{-1}=w_{5^2}^{n_0}\gamma$, for some $n,n_0\in \{0,1\}, $ $w_{d^\prime}\in W$ and $\gamma\in \Gamma_0(5^2{M^\prime})$. If $w_{d^\prime}^{n}\ne w_d$, then $w_{d^\prime}^{-n}\sigma w_d\sigma^{-1}\not\in  \mathrm{PSL}_2(\Z[\frac{1}{5}])$ but $w_{5^2}^{n_0}\gamma\in  \mathrm{PSL}_2(\Z[\frac{1}{5}])$. Hence we have 
 $w_{d}^{-1}\sigma w_d\sigma^{-1}=w_{5^2}^{n_0}\gamma, \ \mathrm{equivalently \ we \ have \ } E:=\Upsilon_5 w_{d}^{-1}\sigma w_d\sigma^{-1}\Upsilon_5^{-1}\in \tilde{\Gamma}_{5{M^\prime}}^5.$
 
 Recall that if $\delta\in \tilde{\Gamma}_{5{M^\prime}}^5$, then we must have 
 $\mathrm{either} \ \delta[0,0]\equiv \delta[1,1]\equiv 0 \pmod 5$ or $\delta[1,0]\equiv \delta[0,1]\equiv 0 \pmod 5$ (note that we always have $\delta[1,0]\equiv 0 \pmod {M^\prime}$). We give a complete proof for the case ${M^\prime}\equiv 1 \pmod 5$. The proofs are similar for the other cases.
 
 Let ${M^\prime}\equiv 1 \pmod 5$. In this case, without loss of generality we can assume that $\sigma=\Upsilon_5^{-1}B_2C_0\Upsilon_5$, and consider a representative $w_d=\frac{1}{\sqrt{d}}\psmat{xd}{y}{5^2{M^\prime}z}{wd}$ such that $xwd^2-5^2{M^\prime}yz=d$. 
 Considering the modulo $5$ reductions, we get
 $$E[0,0]\equiv 2-2w^2d+1 \equiv -2(1+w^2d) \equiv -2(1\pm d) \pmod 5, \ E[0,1]\equiv 4 + 2 -  w^2 d \equiv 1-w^2d \equiv 1-(\pm d) \pmod 5,$$
 $$E[1,0] \equiv -x^2 d + 1 \equiv 1-x^2d \equiv 1-(\pm d) \pmod 5, \ E[1,1]\equiv -2 x^2 d  + 2 +1 \equiv -2(1+x^2d)\equiv -2(1\pm d) \pmod 5.$$
 From the last relations it is easy to see that either 
 $E[0,0]\equiv E[1,1]\equiv 0 \pmod 5$ or $E[1,0]\equiv E[0,1]\equiv 0 \pmod 5$ if and only if $d\equiv \pm 1 \pmod 5.$
 
 Therefore for ${M^\prime}\equiv 1 \pmod 5$, we have $\sigma w_d\sigma^{-1}\in \langle \Gamma_0(5^2{M^\prime}), W\rangle$ if and only if $d\equiv \pm 1 \pmod 5$. The result follows.
\end{proof}

\begin{cor}\label{normalizer of 25 contained in normalizer of 25, u2,...,uk}
Let $N, u_2,\ldots,u_k\in \N$ such that $4,9\nmid N$, $u_i^2||N$ and $5\nmid u_i$ for $i\in \{2,\ldots, k\}$. Then 
$$\mathcal{N}(\langle \Gamma_0(N), w_{25})\subseteq \mathcal{N}(\langle \Gamma_0(N), w_{25}, w_{u_2^2}, \ldots, w_{u_k^2}\rangle).$$
\end{cor}
\begin{proof}
Recall that $w_{25}, w_{e,N}\in \mathcal{N}(\langle \Gamma_0(N), w_{25},  w_{u_2^2}, \ldots, w_{u_k^2}\rangle)$ for $e||\frac{N}{25}$. Any element of $\langle \Gamma_0(N), w_{25},  w_{u_2^2}, \ldots, w_{u_k^2}\rangle \char`\\ \Gamma_0(N)$ can be written in the form $w_{25^{m}d^2}\gamma$ for some $\gamma\in \Gamma_0(N)$, $m\in \{0,1\}$ and $d||\mathrm{lcm}(u_2,\ldots, u_k)$. Let $\sigma \in \{\Upsilon_5^{-1} B_jC_0\Upsilon_5, \Upsilon_5^{-1}B_0C_i\Upsilon_5\}$, where $i,j$ are defined as in Corollary \ref{cor3.6}. Since $\frac{25^{m}d^2}{(25,25^{m}d^2)}\equiv \pm 1\pmod 5$, by Lemma \ref{The order 3 element commutes with Atkin-Lehner operators} we get $\sigma w_{25^{m}d^2}\sigma^{-1}\in \langle \Gamma_0(N), w_{25},  w_{u_2^2}, \ldots, w_{u_k^2}\rangle$. Thus $\sigma\in \mathcal{N}(\langle \Gamma_0(N), w_{25},  w_{u_2^2}, \ldots, w_{u_k^2}\rangle)$. Now the result follows from Corollary \ref{exact normalizer with only 5 square Atkin-Lehner}.
\end{proof}
We are now ready to compute the normalizer of $\langle \Gamma_0(N), w_{25}, w_{u_2^2}, \ldots, w_{u_k^2}\rangle$.

\begin{prop}\label{Theorem 5.11}
Let $N, u_2,\ldots,u_k\in \N$ such that $4,9\nmid N$, $u_i^2||N$ and $5\nmid u_i$ for $i\in \{2,\ldots, k\}$. Then
$$\mathcal{N}(\langle \Gamma_0(N), w_{25}, w_{u_2^2}, \ldots, w_{u_k^2}\rangle) = \mathcal{N}(\langle \Gamma_0(N), w_{25}\rangle).$$
\end{prop}
\begin{proof}
By Corollary \ref{normalizer of 25 contained in normalizer of 25, u2,...,uk} we know that $\mathcal{N}(\langle \Gamma_0(N), w_{25})\subseteq \mathcal{N}(\langle \Gamma_0(N), w_{25}, w_{u_2^2}, \ldots, w_{u_k^2}\rangle)$. 
We now prove the other inclusion.
For simplicity of notation, we write $u:=\mathrm{lcm}(u_2,\ldots, u_k)$ and assume that $u>1$. 

Recall that $(5,u)=1$. By { Theorem \ref{MainThm}} we know that 
$\mathcal{N}(\langle \Gamma_0(N), w_{25},  w_{u_2^2}, \ldots, w_{u_k^2}\rangle) \subseteq \Upsilon_{5u}^{-1} \Gamma_0^*(M)\Upsilon_{5u},$
where $M:=\frac{N}{25u^2}$. As discussed in the beginning of \S \ref{section exact normalizer computation}, it suffices to compute the elements $\Upsilon_{5u}^{-1} \sigma \Upsilon_{5u}$ with $\sigma\in \Gamma_0(M)$ such that $\Upsilon_{5u}^{-1} \sigma \Upsilon_{5u}\in \mathcal{N}(\langle \Gamma_0(N), w_{25}, w_{u_2^2}, \ldots, w_{u_k^2}\rangle)$.
Let $\sigma:=\psmat{a}{b}{c}{d}\in \Gamma_0(M)$ such that $\Upsilon_{5u}^{-1} \sigma \Upsilon_{5u} \in \mathcal{N}(\langle \Gamma_0(N), w_{25}, w_{u_2^2}, \ldots, w_{u_k^2}\rangle)$. Note that for such $\sigma$, we have 
$$\sigma\in \Upsilon_{5u}\mathcal{N}(\langle\Gamma_0(N), w_{25}, w_{u_2^2}, \ldots, w_{u_k^2}\rangle)\Upsilon_{5u}^{-1}=\mathcal{N}(\langle \Gamma_{(5,u_1,\ldots, u_k)}, \Gamma_5, \Gamma_{u_2}, \ldots, \Gamma_{u_k}\rangle).$$
By Proposition \ref{ac and bd divisible by lcm/5 Theorem}, we know that $ac\equiv bd \equiv 0 \pmod u$.

\textbf{Claim}: We now prove that $\gamma:= \Upsilon_{5u}^{-1} \sigma\Upsilon_{5u}=\psmat{a}{\frac{b}{5u}}{5uc}{d}\in \mathcal{N}(\langle \Gamma_0(N), w_{25}\rangle).$ 

Since $(a,b)=(c,d)=1$, the condition $ac\equiv bd \equiv 0 \pmod u$ implies $ab\equiv cd \equiv 0 \pmod u$.

Since $\gamma\in \mathcal{N}(\langle \Gamma_0(N), w_{25}, w_{u_2^2}, \ldots, w_{u_k^2}\rangle)$, for any $\tilde{\gamma}\in \langle \Gamma_0(N), w_{25}\rangle$ we have
$\gamma \tilde{\gamma} \gamma^{-1}\in \langle \Gamma_0(N), w_{25}, w_{u_2^2}, \ldots, w_{u_k^2}\rangle$. Note that  $\alpha \in \langle \Gamma_0(N), w_{25}, w_{u_2^2}, \ldots, w_{u_k^2}\rangle \cap  \mathrm{PSL}_2(\Z[\frac{1}{5}])$ if and only if $\alpha \in \langle \Gamma_0(N), w_{25}\rangle$.
 We prove the claim by showing that $\gamma \tilde{\gamma} \gamma^{-1}\in  \mathrm{PSL}_2(\Z[\frac{1}{5}])$ for any $\tilde{\gamma}\in \langle \Gamma_0(N), w_{25}\rangle$.

First assume that $\gamma_1\in w_{25}\Gamma_0(N)$, i.e., $\gamma_1:=\psmat{5x}{\frac{y}{5}}{\frac{N}{5}z}{5w}$ with $x,y,x,w\in \Z$ such that $25xw-\frac{N}{25}yz=1$. 
Then
\begin{align*}
E_1[0,0]:=-acuy + 5adx - 5bcw + bdMuz, \ &
E_1[0,1]:=\frac{1}{5}a^2y - \frac{ab}{u}(x-w) - \frac{1}{5}Mb^2z,\\
E_1[1,0]:=-5c^2u^2y + 25cdux - 25cduw + 5Mu^2d^2z, \ &
E_1[1,1]:=acuy + 5adw - 5bcx - Mubdz, \ \mathrm{where} \ E_1:=\gamma \gamma_1 \gamma^{-1}.
\end{align*}
Since $ab \equiv 0 \pmod u$, we get $E_1\in  \mathrm{PSL}_2(\Z[\frac{1}{5}])$. Therefore $E_1\in \langle \Gamma_0(N), w_{25}, w_{u_2^2}, \ldots, w_{u_k^2}\rangle \cap  \mathrm{PSL}_2(\Z[\frac{1}{5}])$. Thus we obtain $E_1\in \langle \Gamma_0(N), w_{25}\rangle$.

Now consider $\gamma_2:=\psmat{x}{y}{Nz}{w}\in \Gamma_0(N)$ and let $E_2:=\gamma \gamma_2 \gamma^{-1}$. Then
\begin{align*}
E_2[0,0]:=-5acuy + adx - bcw + 5bdMuz, \ &
E_2[0,1]:=a^2y - \frac{1}{5}\frac{ab}{u}(x-w) - b^2Mz,\\
E_2[1,0]:=-25c^2u^2y + 5cdux - 5cduw + 25d^2Mu^2z, \ &
E_2[1,1]:=5acuy + adw - bcx - 5bdMuz.
\end{align*}
Since $ab \equiv 0 \pmod u$, we get $E_2\in  \mathrm{PSL}_2(\Z[\frac{1}{5}])$. Therefore $E_2\in \langle \Gamma_0(N), w_{25}, w_{u_2^2}, \ldots, w_{u_k^2}\rangle \cap  \mathrm{PSL}_2(\Z[\frac{1}{5}])$. Thus we obtain $E_2\in \langle \Gamma_0(N), w_{25}\rangle$.

Hence we conclude that $\gamma \langle \Gamma_0(N), w_{25}\rangle \gamma^{-1}\subseteq \langle \Gamma_0(N), w_{25}\rangle$, i.e., $\gamma\in \mathcal{N}(\langle \Gamma_0(N), w_{25}\rangle)$. Consequently, we obtain that
$\mathcal{N}(\langle \Gamma_0(N), w_{25}, w_{u_2^2}, \ldots, w_{u_k^2}\rangle) \subseteq \mathcal{N}(\langle \Gamma_0(N), w_{25}\rangle)$. 
The result follows.
\end{proof}
\subsection{Exact normalizer of $\langle \Gamma_0(N), W\rangle$ for arbitrary subgroup $W$}
Let $N\in \N$ such that $4,9\nmid N$ and $W$ be a subgroup generated by certain Atkin-Lehner involutions. Then we can find positive integers $u_1,u_2\ldots, u_k$ and $v_{k+1}, \ldots, v_n$ such that 
\begin{equation}\label{suitable representation of quotient group}
\langle \Gamma_0(N), W\rangle=\langle \Gamma_0(N), w_{u_1^2}, \ldots, w_{u_k^2}, w_{v_{k+1}}, \ldots, w_{v_n}\rangle,
\end{equation}
and for any Atkin-Lehner involution $w_d\in \langle \Gamma_0(N), w_{v_{k+1}}, \ldots, w_{v_n}\rangle$, $d$ is not a perfect square\footnote{This can be done as follows: let $H_1:=\langle \Gamma_0(N),W\rangle/\Gamma_0(N)$ and $H_2:=\langle \Gamma_0(N), w_{u^2}: w_{u^2}\in W\rangle/\Gamma_0(N)$. Then the generators of $H_2$ will give the $u_i$'s, and the non-trivial generators of $H_1/H_2$ will give the $v_j$'s.}.

It is well known that $\mathcal{N}(\langle \Gamma_0(N), W\rangle) \supseteq \Gamma_0^*(N).$ Moreover,
by Theorem \ref{MainThm}, we know that 
\begin{equation}\label{eq 5.50}
\mathcal{N}(\langle \Gamma_0(N), W\rangle)\leq \mathcal{N}(\langle \Gamma_0(N), w_{u_1^2}, \ldots, w_{u_k^2}\rangle).
\end{equation}
 If $w_{25} \notin W$, then $w_{25}\notin \langle \Gamma_0(N), w_{u_1^2}, \ldots, w_{u_k^2}\rangle$. Consequently, by Corollary \ref{exact normalizer for u gt 5} we have $\mathcal{N}(\langle \Gamma_0(N), w_{u_1^2}, \ldots, w_{u_k^2}\rangle)=\Gamma_0^*(N)$. Therefore using \eqref{eq 5.50} we conclude that $\mathcal{N}(\langle \Gamma_0(N), W\rangle) = \Gamma_0^*(N).$ Thus obtain the following theorem.
\begin{thm}\label{Complete normalizer without 25 theorem}
Let $N\in \N$ and $W$ be a subgroup generated by certain Atkin-Lehner involutions such that $4,9\nmid N$ and $w_{25}\notin W$. Then 
$\mathcal{N}(\langle \Gamma_0(N), W\rangle) = \Gamma_0^*(N).$
\end{thm}
We now study the case where $w_{25}\in W$.   More precisely, we prove the following theorem.
\begin{thm}\label{Complete normalizer with 25 theorem}
Let $N\in \N$ and $W$ be a subgroup generated by certain Atkin-Lehner involutions such that $4,9\nmid N$ and $w_{25}\in W$.
\begin{enumerate}
\item If there exists $w_d\in W$ such that $\frac{d}{(25,d)}\not\equiv \pm 1 \pmod 5$, then $\mathcal{N}(\langle \Gamma_0(N), W\rangle) = \Gamma_0^*(N)$.
\item If $\frac{d}{(25,d)}\equiv \pm 1 \pmod 5$ for all $w_d\in W$, then $\mathcal{N}(\langle \Gamma_0(N), W\rangle) = \langle \Gamma_0^*(N), \Upsilon_5^{-1} B_jC_0\Upsilon_5, \Upsilon_5^{-1}B_0C_i\Upsilon_5\rangle$ where $B_j:=\psmat{\frac{N}{25} j+1}{-j}{-{\frac{N}{25}}}{1}$, $C_i:=\psmat{1}{i}{0}{1}$, $0\leq j,i \leq 4$ such that  ${\frac{N}{25}}j\equiv 2\pmod 5$ and $i+j\equiv 0 \pmod 5$. 
{ Moreover, the subgroup $\langle \Upsilon_5^{-1} B_jC_0\Upsilon_5= (\Upsilon_5^{-1}B_0C_i\Upsilon_5)^{-1} \rangle$ has order $3$ in $\mathcal{N}(\langle \Gamma_0(N), W\rangle)/\langle \Gamma_0(N), W\rangle$.}             
\end{enumerate}
\end{thm}
\begin{proof}
Recall that $\mathcal{N}(\langle \Gamma_0(N), W\rangle) \supseteq \Gamma_0^*(N).$ Since $w_{25}\in W$, without loss of generality in \eqref{suitable representation of quotient group} we can assume that $u_1=5$ and $5\nmid (\prod_{i^\prime=2}^k u_{i^\prime}\cdot \prod_{j^\prime=k+1}^n v_{j^\prime})$. By Proposition \ref{Theorem 5.11} and Corollary \ref{exact normalizer with only 5 square Atkin-Lehner} we get
\begin{equation}\label{eq 5.51}
\mathcal{N}(\langle \Gamma_0(N), w_{25}, w_{u_2^2}, \ldots, w_{u_k^2}\rangle) = \langle \Gamma_0^*(N), \Upsilon_5^{-1} B_jC_0\Upsilon_5, \Upsilon_5^{-1}B_0C_i\Upsilon_5\rangle,
\end{equation}
 where $B_j:=\psmat{\frac{N}{25} j+1}{-j}{-{\frac{N}{25}}}{1}$, $C_i:=\psmat{1}{i}{0}{1}$, $0\leq j,i \leq 4$ such that  ${\frac{N}{25}}j\equiv 2\pmod 5$ and $i+j\equiv 0 \pmod 5$. 

Let $\sigma \in \{\Upsilon_5^{-1} B_jC_0\Upsilon_5, \Upsilon_5^{-1}B_0C_i\Upsilon_5\}$. If there exists $w_d\in W$ such that $\frac{d}{(25,d)}\not\equiv \pm 1 \pmod 5$, then from Lemma \ref{The order 3 element commutes with Atkin-Lehner operators} we get $\sigma w_d \sigma^{-1}\notin \langle \Gamma_0(N), W\rangle$. Consequently from \eqref{eq 5.50} and \eqref{eq 5.51} we obtain $\mathcal{N}(\langle \Gamma_0(N), W\rangle) = \Gamma_0^*(N)$. This proves the first part.

On the other hand if $\frac{d}{(25,d)}\equiv \pm 1\pmod 5$ for all $w_d\in W$, then from Lemma \ref{The order 3 element commutes with Atkin-Lehner operators} we get $\sigma w_d \sigma^{-1}\in \langle \Gamma_0(N), W\rangle$ for all $w_d\in \langle \Gamma_0(N), W\rangle$. Thus $\sigma \in \mathcal{N}(\langle \Gamma_0(N), W\rangle)$. Moreover, it is easy to check that $\sigma$ has order $3$ in $\mathcal{N}(\langle \Gamma_0(N), W\rangle)/\langle \Gamma_0(N), W\rangle$ and $\Upsilon_5^{-1} B_jC_0\Upsilon_5, \Upsilon_5^{-1}B_0C_i\Upsilon_5$ are inverse of each other in $\mathcal{N}(\langle \Gamma_0(N), W\rangle)/\langle \Gamma_0(N), W\rangle$.
Now the second part follows from \eqref{eq 5.50} and \eqref{eq 5.51}.
\end{proof}

%% file: Section3FinalVersion_Revised.tex
\section{On the modular { automorphisms} of order 3 of { $X_0(N)/W_N$}}

{ For $N\in \N$ and a subgroup $W_N\leq B(N)$, consider the quotient curve $X_0(N)/W_N$.
Clearly we have}
$$B(N)/W_N\leq \mathcal{N}(\Gamma_0(N)+W_N)/(\Gamma_0(N)+W_N)\leq \mathrm{Aut}(X_0(N)/W_N)$$
where $\mathcal{N}(\Gamma_0(N)+W_N)/(\Gamma_0(N)+W_N)$ is the modular { automorphism} group of $X_0(N)/W_N$.

Assuming $4,9\nmid N$, { by Theorem \ref{Complete normalizer without 25 theorem} and Theorem \ref{Complete normalizer with 25 theorem} we know that the modular automorphism group of $X_0(N)/W_N$ is exactly $B(N)/W_N$, except for the case $w_{25}\in W_N$ and $\frac{d}{(25,d)}\equiv \pm 1 \pmod 5$ for all $w_d\in W_N$; and in such situation we have
 $$\mathcal{N}(\Gamma_0(N)+W_N)/(\Gamma_0(N)+W_N)=\langle \Gamma_0^*(N), \Upsilon_5^{-1} B_jC_0\Upsilon_5, \Upsilon_5^{-1}B_0C_i\Upsilon_5\rangle/(\Gamma_0(N)+W_N),$$ where $0\leq j,i \leq 4$ such that  ${\frac{N}{25}}j\equiv 2\pmod 5$ and $i+j\equiv 0 \pmod 5$.}

 
 In this section we restrict to $N=25 M$ where $M$ is square-free, $(5,M)=1$, and a subgroup $W_N\leq B(N)$ of the form $\langle w_{25},w_{v_2},\ldots,w_{v_n}\rangle$ with $v_l||M$ and $v_l\equiv\pm1 \pmod 5$ for all $l\in \{2,\ldots,n\}$. 
Consider the order $3$ element $\sigma_{M}:= \Upsilon_5^{-1}B_j C_0\Upsilon_5$ in $\mathcal{N}(\Gamma_0(25M)+W_{25M})/(\Gamma_0(25M)+W_{25M})$.

\begin{lema} Under the assumptions and notations in this section, $\mathcal{N}(\Gamma_0(25M)+W_{25M})/(\Gamma_0(25M)+W_{25M})=\langle B(25M)/W_{25M},\sigma_M\rangle$ and $\sigma_M$ has order 3. Then $\sigma_M$ is defined over $\mathbb{Q}(\sqrt{5})$ (as an automorphism of $X_0(N)/W_N$), in particular $\mathrm{Aut}(X_0(N)/W_N)=\mathrm{Aut}_{\mathbb{Q}(\sqrt{5})}(X_0(N)/W_N)$ (where $\mathrm{Aut}_K(X)$ denotes the group of all automorphisms of $X$ defined over the field $K$).
	\end{lema}
\begin{proof} 
Note that the elements of $\textrm{Aut}(X_0(N)/W_N)$ can be thought of as automorphisms on the Jacobian variety of $X_0(N)/W_N$. 
Let $\infty$ be the cusp at infinity of $X_0(N)/W_N$. Then it is easy to check that $\sigma_M(\infty)$ is not a rational cusp (cf. \cite{Ogg} for the field of definition of the cusps). Therefore $\sigma_M$ is not defined over $\Q$. Now the result follows from the fact that any automorphism of the Jacobian is defined over the compositum of the quadratic fields with discriminant $D$ whose square divides $N$ (cf. \cite[Proposition 1.3, Lemma 1.5]{KeMo}).
	\end{proof}


\begin{remark}  Assume $p\equiv 1 \pmod 4$ is a prime, $M$ is a square-free positive integer coprime with $p$, and $W_{p^2M}=\langle w_{p^2}, w_{v_2},\ldots,w_{v_n}\rangle$ with $v_l||M$. 
	Then, by \cite[Lemma 1.5]{KeMo} and \cite[Proposition 1.3]{KeMo} any automorphism of $X_0(p^2M)/W_{p^2M}$ is defined either over $\mathbb{Q}$ or over $\mathbb{Q}(\sqrt{p})$ (the same conclusion is true if $p\equiv 3 \pmod 4$ and the Jacobian variety of $X_0({p^2M})/W_{p^2M}$ does not contain any subvariety with complex multiplication). For a prime $\ell\nmid {p^2M}$ we can reduce the curve $X_0({p^2M})/W_{p^2M}$ modulo $\ell$, and denote such curve by $\mathcal{X}_0({p^2M})/W_{p^2M}\otimes\mathbb{F}_{\ell}$. Then we have an injection
	$$\mathrm{Aut}(X_0({p^2M})/W_{p^2M})\hookrightarrow \mathrm{Aut}_{\mathbb{F}_{\ell^2}}(\mathcal{X}_0({p^2M})/W_{p^2M}\otimes\mathbb{F}_{\ell}),$$
	and using Magma in many cases (with small genus) we can compute the automorphism group over the finite field $\mathbb{F}_{\ell^2}$, via the instruction
	\begin{verbatim}
	Automorphisms(ChangeRing(X0NQuotient(p^2*M,[p^2,v_2,\ldots,v_n]),GF(\ell^2))).
	\end{verbatim}
	Consequently, we have an upper bound for the order of the automorphism group, and a lower bound is given by the order of the modular automorphism group. For example, using Magma we obtain $|\mathrm{Aut}_{\F_4}(X_0(275)/\langle w_{25}\rangle)|=6$, consequently we get $|\mathrm{Aut}(X_0(275)/\langle w_{25})\rangle|\leq 6$. Furthermore by Theorem \ref{Complete normalizer with 25 theorem}, we have $|\mathrm{Aut}(X_0(275)/\langle w_{25})\rangle|\geq 6$. Therefore we conclude that $|\mathrm{Aut}(X_0(275)/\langle w_{25})\rangle|= 6=|\mathcal{N}(\Gamma_0(275)+\langle w_{25}\rangle)/(\Gamma_0(N)+\langle w_{25}\rangle)|$, i.e., $\mathrm{Aut}(X_0(275)/\langle w_{25})\rangle=\mathcal{N}(\Gamma_0(275)+\langle w_{25}\rangle)/(\Gamma_0(275)+\langle w_{25}\rangle)$.
\end{remark}

Now consider $\sigma_M$ as an element of $X_0(25M)/\langle w_{25}\rangle$. We now give a theoretical explanation of the fact that for positive integers $v_2,\ldots,v_n$ with $v_l||M$ and $v_l\equiv \pm 1 \pmod 5$, $\sigma_M$ induces an automorphism of order 3 on $X_0(25 M)/\langle w_{25},w_{v_2},\ldots,w_{v_n}\rangle$.


We write the $\mathbb{Q}$-decomposition of the Jacobian of $X_0(25M)/\langle w_{25}\rangle$ by:
\begin{equation}\label{eq 4.1 new}
\mathrm{Jac}(X_0(25 M)/\langle w_{25}\rangle)\sim_{\mathbb{Q}} \prod_{m=1}^s A_{f_m}^{n_m},
\end{equation}
where $f_m$ is a newform of level $N_m$ (with $N_m|25M$) such that $w_{25}$ acts as $+1$ on $f_m$ if $25|N_m$.
Since $\mathrm{Aut}(X_0(25 M)/\langle w_{25}\rangle)$ has an automorphism defined over $\mathbb{Q}(\sqrt{5})$ but not over $\mathbb{Q}$, there exist { $f_{l_1},f_{l_2}$} (in \eqref{eq 4.1 new}) such that $A_{f_{l_1}}\sim_{\mathbb{Q}(\sqrt{5})}A_{f_{l_2}}$ with $f_{l_2}=f_{l_1}\otimes\chi_5$ (where $\chi_5$ is the quadratic Dirichlet character associated to $\Q(\sqrt{5})$).
  It is well-known
that the abelian variety $A_f$ is simple over $\overline{\Q}$ if and only if $f$ does not have any inner twist,
i.e. there is no quadratic Dirichlet character $\chi$ such that $f\otimes\chi$ is a Galois conjugated of
$f$. This condition amounts to say that $\mathrm{End}_{\Q}(A_f) = \mathrm{End}_{\overline{\Q}}(A_f )$.
Assume that $f_{l_1}$ and $f_{l_2}$ (appearing in \eqref{eq 4.1 new}) do not have any inner twist (in particular $l_1\ne l_2$). If $A_{f_{l_1}}$ and $A_{f_{l_2}}$ are isogenous over $\overline{\Q}$
but not over $\mathbb{Q}$, then there exists a Dirichlet character $\chi$ such that $A_{f_{l_2}} = A_{f_{l_1}}\otimes \chi$ (cf. \cite[Proposition 4.2]{GJU12}), and if
$\chi$ is the quadratic Dirichlet character attached to the quadratic number field $\Q(\sqrt{5})$, then there is an isogeny (defined over $\Q(\sqrt{5})$) between the abelian varieties $A_{f_{l_2}}$ and $A_{f_{l_1}}$. 

Therefore assuming that all the $f_m$'s appearing in \eqref{eq 4.1 new} have no inner twist, we have that the modular automorphisms of order 3 in Corollary \ref{cor3.6} are coming from matrices (acting on the canonical model obtained using the cusp forms appearing in \eqref{eq 4.1 new}) defined over $\mathbb{Q}(\sqrt{5})$ (such matrices consist of blocks corresponding to the $\mathbb{Q}(\sqrt{5})$-isogeny factors $A_{f_{l_1}}^{n_{l_1}}\times A_{f_{l_2}}^{n_{l_2}}\sim_{\Q(\sqrt{5})} A_{f_{l_1}}^{n_{l_1}+n_{l_2}}$ where $f_{l_2}=f_{l_1}\otimes \chi_5$ and $f_{l_1}\neq f_{l_2}$). 

In order that the modular automorphism $\sigma_M$ of order 3 of $X_0(25 M)/\langle w_{25}\rangle$ descends to an order 3 automorphism of $X_0(25 M)/\langle w_{25},w_{v_2},\ldots,w_{v_n}\rangle$, a sufficient condition is that for any quadratic twist $f_{l_1}\otimes\chi_5=f_{l_2}$ (in \eqref{eq 4.1 new}) where the action of $\sigma_M$ is non-trivial on the $\Q(\sqrt{5})$-isogeny factor $A_{f_{l_1}}^{n_{l_1}+n_{l_2}}$, the Atkin-Lehner involution $w_{v_l}$ should act with the same sign on $A_{f_{l_1}}$ and  $A_{f_{l_1}}$ (cf. \cite[Lemma 18]{BaTa24} for more detail). 

We recall the following result of Atkin-Lehner in \cite[p.156]{AtLe} concerning quadratic twists:
\begin{lema}\label{lemma 4.3} 
Let $p$ be a prime, $M^\prime\in \N$ with $(p,M^\prime)=1$, and $\chi_p$ be the quadratic Dirichlet character associated to $\Q(\sqrt{p})$. If $f$ is a newform for $\Gamma_0(M^\prime)$ or $\Gamma_0(pM^\prime)$, then $f\otimes \chi_p$ is a newform for $\Gamma_0(p^2M^\prime)$. Furthermore 
	\begin{itemize}
		\item for $d||M^\prime$ we have $f\otimes\chi_p|{w_d}=\left(\frac{d}{p}\right) \epsilon_d(f) (f\otimes\chi_p)$, where $f|w_d= \epsilon_d(f) f$, and $\left(\frac{d}{p}\right)$ denotes the Kronecker symbol.
		\item $f\otimes\chi_p|{w_{p^2}}=\left(\frac{-1}{p}\right) f\otimes\chi_p$
	\end{itemize} 
If $f$ is a newform for $\Gamma_0(p^2M^\prime)$ and $f\otimes\chi_p$ is also a newform for $\Gamma_0(p^2M^\prime)$, then for any $d||M^\prime$ we have  $f\otimes\chi_p|{w_d}=\left(\frac{d}{p}\right) \epsilon_d(f) f\otimes\chi_p$.
\end{lema}
We recall that (cf. \cite{Ogg}) if $f_{m_1}\otimes\chi_p=f_{m_2}$ where $f_{m_1}$ and $f_{m_2}$ are newforms of levels $M_1$ and $M_2$ respectively, both dividing $p^2 M$ with $(M,p)=1$, then $M_1=p^2M^\prime$ or $M_2=p^2M^\prime$ for a natural number $M^\prime|M$. Therefore, the level of the other quadratic twisted modular form involved is $M^\prime$, $pM^\prime$ or $p^2M^\prime$.

\begin{cor}\label{cor4.4}
	Consider $N= 5^2 M$, where $M$ is a square-free positive integer with $(5,M)=1$. Assume that the Jacobian of $X_0(N)/\langle w_{25}\rangle$ has no inner twist, and for each quadratic twist $f_{m_1}\otimes\chi_5=f_{m_2}$ with ${m_1}\neq {m_2}$ (where $A_{f_{m_1}}$ and $A_{f_{m_2}}$ are distinct $\Q$-isogeny factors of $\mathrm{Jac}(X_0(N)/\langle w_{25}\rangle)$) the conductor of $f_{m_1}$ or $f_{m_2}$ is equal to $N$. Then, an Atkin-Lehner involution $w_d$ (with $d||M$) acts exactly by the same sign on $A_{f_{m_1}}$ and $A_{f_{m_2}}$ iff $\left(\frac{d}{5}\right)=1$ iff $d\equiv \pm 1(\mod 5)$.
\end{cor}
Recall that any non-trivial $w_d\in B(p^2M)/W_{p^2M}$ ($p\equiv 1 \pmod 4$ is a prime and $p\nmid M$) acts by $\pm \mathrm{id}$ on each $\Q$-isogeny factor of the Jacobian of $X_0({p^2M})/W_{p^2M}$. Furthermore, if it acts exactly with the same sign on all distinct $\Q$-isogeny factors that become isogenous over $\Q(\sqrt{p})$,
  then any $w\in \textrm{Aut}_{\Q(\sqrt{p})}(X_0(p^2M)/W_{p^2M})\setminus (B(p^2M)/W_{p^2M})$  induces a non-trivial automorphism of $X_0(p^2M)/\langle W_{p^2M}, w_d\rangle$ over $\Q(\sqrt{p})$. (cf. \cite[Lemma 18]{BaTa24}). 
\begin{cor}
		Consider $N= 5^2 M$, where $M$ is a square-free positive integer with $(5,M)=1$.
Assume that the Jacobian of $X_0(N)/\langle w_{25}\rangle$ has no inner twist, and for each quadratic twist $f_{m_1}\otimes\chi_5=f_{m_2}$ with ${m_1}\neq {m_2}$ (where $A_{f_{m_1}}$ and $A_{f_{m_2}}$ are distinct $\Q$-isogeny factors of $\mathrm{Jac}(X_0(N)/\langle w_{25}\rangle)$) the conductor of $f_{m_1}$ or $f_{m_2}$ is equal to $N$.	
	 Let $\sigma_M$ an element of order 3 in
$\mathcal{N}(\Gamma_0(N)+\langle w_{25}\rangle)/(\Gamma_0(N)+\langle w_{25}\rangle)$. Then for positive integers $v_2,\ldots,v_n$ with $v_l||M$ and $v_l\equiv \pm 1 \pmod 5$, 
$\sigma_M$ induces an automorphism of order 3 on $X_0(25 M)/\langle w_{25},w_{v_2},\ldots,w_{v_n}\rangle$.
\end{cor}


\begin{remark} 
{ By Lemma \ref{lemma 4.3} it is easy to see that
when $M$ is a prime, the assumption that the conductor of $f_{m_1}$ or $f_{m_2}$ is equal to $N=5^2 M$ is always true.}
\end{remark}


%% file: Normalizer_Revised.bbl
\begin{thebibliography}{99}
\bibitem[AtLeh70]{AtLe}{ Atkin, A. O. L.; Lehner, J.}: Hecke operators on $\Gamma \sb{0}(m)$. Math. Ann. 185 (1970), 134--160.
\bibitem[AkSi90]{AkSin} { Akbas, M.; Singerman, D.}: The normalizer of $\Gamma_0(N)$ in ${\rm PSL}(2,{\bf R})$. Glasgow Math. J. 32 (1990), no. 3, 317--327.
\bibitem[Ba08]{Ba} { Bars, F.}: The group structure of the normalizer of $\Gamma_0(N)$ after Atkin-Lehner. Comm. Algebra 36 (2008), no. 6, 2160--2170.
\bibitem[BaDa24]{BaTa24} { Bars, F.; Dalal, T.}: On the automorphism group of quotient modular curves. Journal of Algebra (2025). {\url{https://doi.org/10.1016/j.jalgebra.2025.02.037}}
\bibitem[BaGo21]{BG21}
{ Bars, F.; Gonz\'alez, J.}:
The automorphism group of the modular curve $X_0^*(N)$ with square-free level.
Trans. Amer. Math. Soc. 374 (2021), no. 8, 5783--5803.
\bibitem[Con96]{Conway} { Conway, J. H.}: Understanding groups like $\Gamma_0(N)$. Groups, difference sets, and the Monster (Columbus, OH, 1993), 327--343, Ohio State Univ. Math. Res. Inst. Publ., 4, de Gruyter, Berlin, 1996.
\bibitem[Do16]{Do16}{ Dose, V.}: On the automorphisms of the nonsplit Cartan modular curves of prime level. Nagoya Math. J. 224 (2016), no.1, 74--92.
\bibitem[DLM22]{DLM22}{ Dose, V.; Lido, G.; Mercuri, P.}: Automorphisms of Cartan modular curves of prime and composite level. Algebar Number Theory 16(6), (2022), 1423--1461.
\bibitem[GaJiU12]{GJU12}{ Gonz\'{a}lez, J.; Jim\'{e}nez-Urroz, J.}: The Sato-Tate distribution and
the values of Fourier coefficients of modular newforms. Exp. Math.,
21(1):84–102, 2012.
\bibitem[KenMom88]{KeMo} { Kenku,M.A. ; Momose, F.}:Automorphism groups of the modular curves $X_0(N)$. Compositio Math. 65 (1988), no. 1, 51--80.
\bibitem[Lan01]{Lang} { Lang, Mong-Lung}: Normalizers of the congruence subgroups of the Hecke groups $G_4$ and $G_6$. J. Number Theory 90 (2001), no. 1, 31--43.
\bibitem[Ogg73]{Ogg}{ Ogg, A.P.}: Rational points on certain elliptic modular curves. Analytic number theory (Proc. Sympos. Pure Math., Vol. XXIV, St. Louis Univ., St. Louis, Mo., 1972), pp. 221--231, Proc. Sympos. Pure Math., Vol. XXIV, Amer. Math. Soc., Providence, RI, 1973.
\end{thebibliography}
